\documentclass[11pt]{amsart}

\usepackage{amssymb,amscd,amsmath}
\usepackage[all]{xy}
\CompileMatrices

\numberwithin{equation}{section}

\title{Proof of Swiss Cheese Version of Deligne's Conjecture}

\author{V.A. Dolgushev, D.E. Tamarkin, and B.L. Tsygan}

\address{}
\usepackage[dvips]{graphics}
\usepackage{graphicx}
\usepackage{psfrag}
\input xy
\xyoption{all}

\newcommand{\Ab}{{\bf Ab}}
\newcommand{\Top}{{\bf Top}}
\newcommand{\sSets}{{\bf sSets}}
\newcommand{\De}{{\bigtriangleup}}
\newcommand{\D}{{\Delta}}

\newcommand{\pa}{{\partial}}
\newcommand{\J}{{\mathcal J}}
\newcommand{\vs}{\sigma}
\newcommand{\al}{\alpha}
\newcommand{\ga}{\gamma}
\newcommand{\de}{\delta}
\newcommand{\La}{\Lambda}
\newcommand{\Te}{\Theta}
\DeclareMathOperator{\SC}{{\text{SC}}}

\newcommand{\I}{{\mathcal I}}
\newcommand\OO{{\text{\O}}}

\newcommand\cO{\mathcal{O}}
\newcommand\bbR{{\mathbb R}}
\newcommand\gf{{\mathbf{k}}}
\DeclareMathOperator\full{{\mathbf{full}}}
 \newcommand\bu{{\mathbf 1}}
 \newcommand\ts{\tilde{s}}
 \newcommand\wt{\,\widetilde{\mathbf t}\,}
 \newcommand\tI{\widetilde{I}}
 
 \newcommand\Ord{{\rm O r d}}
 \newcommand\triv{{\mathbf{ triv}}}
 \newcommand\GJ{{\mathbf{ GJ}}}

 \newcommand\FM{{\mathbf{FM}}}
 
 \newcommand\FN{{\mathbf{ FN}}}
\renewcommand\a{{\mathfrak a}}
\renewcommand\c{{\mathfrak c}}
\renewcommand\u{{\mathfrak u}}

\newcommand\C{{\mathcal C}}

\renewcommand\top{{\mathbf{ top}}}
\renewcommand\S{{\mathcal S}}
\newcommand\Id{\text{Id}}
\newcommand\chrom{{X\rho}}
\renewcommand\vec{\overrightarrow}
\newcommand\seq{{\mathbf{ seq}}}
\newcommand\se{{\mathbf{ s}}\text{\O}}
\newcommand\sO{\se}
\newcommand\complexes{{\mathbf{ complexes}}}
\renewcommand\t{{\mathbf t}}
\newcommand\bighat{\tilde}

\newcommand\nat{{\Bbb N}}
\newcommand\bbZ{{\Bbb Z}}
\renewcommand\I{{\mathcal I}}

\DeclareMathOperator\des{{\mathbf{ des}}}
\DeclareMathOperator \sym{{\mathbf{ sym}}}
\newcommand\op{\OO}
\newcommand\Sets{{\mathbf{Sets}}}

\DeclareMathOperator\colim{{\mathbf{
colim}}}
\DeclareMathOperator\hocolim{{\mathbf{
hocolim}}}
\newcommand\pt{{\mathbf{ pt}}}

\newcommand\sco{|\se|}
\newcommand\scoR{{\mathbf{ braces}}}
\newcommand\scseq{{\stackrel{{\mathbf{ SC}}}{\mathbf{
\seq}}}}
\newcommand\as{{\mathbf{ assoc}}}

\newcommand\ord{{\mathbf{ ord}}}
\newcommand\R{{\mathcal R}}
\renewcommand\k{{\mathbf k}}
\newcommand\opp{{\text{op}}}
\renewcommand\sc{{|\scseq|}}
\newcommand\scR{{\mathbf{ br}}}
\newcommand\ttrees{\text{\bf 2-trees}}

\newcommand\bul{{\bullet}}

\DeclareMathOperator\conf{{\mathbf{Conf}}}

\begin{document}
\newtheorem{theorem}{Theorem}[section]
\newtheorem{Axiom}[theorem]{Axiom}
\newtheorem{Claim}[theorem]{Claim}
\newtheorem{Conjecture}[theorem]{Conjecture}
\newtheorem{Lemma}[theorem]{Lemma}
\newtheorem{Corollary}[theorem]{Corollary}
\newtheorem{Proposition}[theorem]{Proposition}
\newtheorem{Theorem}[theorem]{Theorem}
\newtheorem{Example}[theorem]{Example}
\newtheorem{Definition}[theorem]{Definition}
\newtheorem{Definition-Proposition}[theorem]
{Definition-Proposition}
\newtheorem{Condition}[theorem]{Condition}

\maketitle

\begin{center}
{\it To the beautiful country of Confoederatio Helvetica.}
\end{center}

\begin{abstract}
For an associative algebra $A$ we consider the pair
``the Hochschild cochain complex $C^{\bul}(A,A)$ and
the algebra $A$''. There is a natural 2-colored operad
which acts on this pair. We show that this operad is
quasi-isomorphic to the singular chain operad
of Voronov's Swiss Cheese operad. This statement
is the Swiss Cheese version of the Deligne conjecture 
formulated by M. Kontsevich in \cite{Km}.
\end{abstract}

~\\
{\it 2000 MSC: 18D50, 16E40.}

\tableofcontents

\section{Introduction}
The interest to various versions \cite{BFN}, \cite{BF}, \cite{HKV},
\cite{RK},
\cite{KS}, \cite{K-Soi1}, \cite{Markl},
\cite{M-Smith}, \cite{Dima-d}, \cite{TZ}, \cite{Bruno}
of the Deligne conjecture on Hochschild complex
is motivated by
generalizations \cite{DTT}, \cite{Cepochki}, \cite{Dima-Proof}, \cite{TT}, 
\cite{Boris} of the famous Kontsevich's formality theorem \cite{K}.
Thus, in recent preprint \cite{K-Soi1} M. Kontsevich
and Y. Soibelman proposed a proof of
the chain version of Deligne's conjecture for
Hochschild complexes of
an $A_{\infty}$-algebra. This is an important
step in proving the formality for the homotopy
calculus algebra of Hochschild (co)chains
\cite{Cepochki}.

Let $A$ be an associative algebra and $C^{\bul}(A,A)$
be the Hochschild cochain complex of $A$\,.
The original version of Deligne's conjecture says that
the operad of natural operations on $C^{\bul}(A,A)$
is quasi-isomorphic to the singular chain operad
of the operad $E_2$ of little discs \cite{Board-V}, \cite{May}.
This statement is not very precise because there are different
choices of what one may call ``the operad of natural operations
on $C^{\bul}(A,A)$\,.'' One may use the so-called minimal
operad of M. Kontsevich and Y. Soibelman \cite{KS} or the
operad of braces \cite{Ezra}, \cite{Gruzia} as in \cite{M-Smith}
and \cite{Sasha1} or the ``big operad'' of M. Batanin and
M. Markl \cite{Misha-Martin}. Due to works of various people
\cite{Misha-Martin}, \cite{BF}, \cite{KS}, \cite{M-Smith},
\cite{dgcat}, and \cite{Sasha1} it is now known that all
these operads are quasi-isomorphic to the singular chain operad
of the operad $E_2$\,.

The topological operad $E_2$ of little discs
admits a natural extension to a 2-colored
topological operad which is called the
Swiss Cheese operad $\SC_2$\,. This operad was
proposed by A. Voronov in \cite{Vor}.

In \cite{Vor} A. Voronov also described
the homology operad $H_{-\bul}(\SC_2)$\,.
More precisely, he showed that
an algebra over the operad $H_{-\bul}(\SC_2)$
is a pair of graded vector spaces
$(V_1, V_2)$, where $V_1$ is a Gerstenhaber
algebra\footnote{In particular, it means that $V_1$ is a
commutative algebra.},
and $V_2$ is an associative algebra equipped
with a module structure over
the commutative algebra $V_1$
\begin{equation}
\label{mod-struc}
V_1 \otimes V_2 \to V_2\,,
\end{equation}
satisfying the following condition
\begin{equation}
\label{uslov}
(u_1 \cdot v_1) \dots (u_n \cdot v_n) =
(u_1 \dots u_n) \cdot (v_1  \dots v_n)\,,
\end{equation}
where $u_i\in V_1$, $v_i \in V_2$, and
for the multiplication of the corresponding
elements we use either the associative
algebra structure
in $V_2$ or the commutative algebra
structure in $V_1$\,.

It is not hard to prove the following
proposition:
\begin{Proposition}
\label{HH-swiss}
If $A$ is an associative algebra
and $HH^{\bul}(A,A)$ is its Hochschild
cohomology then the pair
$(HH^{\bul}(A,A), A)$ forms an algebra over
the operad $H_{-\bul}(\SC_2)$\,.
\end{Proposition}
\begin{proof}
Indeed the associative algebra structure on
$A$ is already given. $HH^{\bul}(A,A)$ is
a Gerstenhaber algebra due to \cite{Ger}.
Finally, to define the module
structure on $A$ over the commutative algebra $HH^{\bul}(A,A)$ we
use the fact that the
zeroth Hochschild cohomology $HH^0(A,A)$
is the center $Z(A)$ of $A$. Namely,
we declare
$$
z \cdot a =
\begin{cases}
z \, a \,, {\rm if} ~ z \in HH^0(A,A) = Z(A)\,, \cr
0\,, ~ {\rm otherwise}\,.
\end{cases}
$$
Equation (\ref{uslov}) is nontrivial only
when $u_i\in HH^0(A,A)$. In this case
the required condition
is automatically satisfied since $u_i$'s are
elements of the center $Z(A)$ of $A$\,.
\end{proof}

In this paper we prove the Swiss Cheese version of
Deligne's conjecture which extends Proposition \ref{HH-swiss}
to the level of cochains.

To formulate this version of Deligne's conjecture we, first,
construct a 2-colored DG operad $\La$ of natural operations on
the pair $(C^{\bul}(A,A) ; A)$\,. Roughly speaking, this
operad is generated by the insertions
of a cochain into a cochain, the cup-product of cochains
and the insertions of elements of the algebra $A$ into
a cochain. The precise description of $\La$ is given in
Section \ref{operad-La}.

The main result of this paper is the following theorem
\begin{Theorem}
\label{main}
The 2-colored DG operad $\La$ of natural operations
on the pair $(C^{\bul}(A,A), A)$ is quasi-isomorphic to
the singular chain operad of Voronov's
Swiss Cheese operad $\SC_2$\,. The induced
action of the homology operad $H_{-\bul}(\SC_2)$
on the pair $(HH^{\bul}(A), A)$ recovers the one
from Proposition \ref{HH-swiss}\,.
\end{Theorem}

We prove this theorem using ideas from \cite{dgcat} and
Batanin's theorem \cite{Bat}
which identifies the homotopy type of Voronov's Swiss Cheese
operad with that of the symmetrization of a contractible cofibrant
Swiss Cheese type 2-operad. The required facts about 2-operads
are reviewed in Sections \ref{review},\ref{screview}

\subsection{Remarks on higher dimensional versions}
Voronov's Swiss Cheese operad
admits the obvious higher dimensional analogue
$\SC_d$ $(d \ge 2)$\,. This operad extends
the operad of $d$-cubes in the same way as the
operad $\SC_2$ extends the operad of
little disks. From this point of view,
Theorem \ref{main}
is a $2$-dimensional case of the following
conjecture formulated by M. Kontsevich in 
\cite{Km}: {\it the DG operad of natural operations
on the pair ``a $d$-algebra\footnote{Recall from \cite{GJ} that a $d$-algebra 
is an algebra over the
homology operad $H_{-\bul}(E_d)$ of the operad of
little $d$-cubes $E_d$\,.} 
and its Hochschild complex''
is quasi-isomorphic to the singular chain operad of
$\SC_{d+1}$\,. } 
In \cite[Section 2.5]{Km} M. Kontsevich also conjectures 
that the Hochschild cochain complex of a $d$-algebra 
is a final object 
in an appropriate category of ``Swiss Cheese algebras". 
In our paper, this question about universality is not addressed.

In \cite{jnkf} J.N.K. Francis showed that an appropriate
deformation complex for a $d$-algebra $A$ is an extension of
its Hochschild complex by $A$\,. In the spirit of this
result the above version of Deligne's conjecture
can be reformulated as follows: {\it the DG operad of natural
operations on the deformation complex of a $d$-algebra
is quasi-isomorphic to the singular chain operad of
$\SC_{d+1}$ }.

~\\
{\bf Notation and conventions.}
We denote by $\k$ the ground field and
by ``(co)chain complexes'' we mean (co)chain complexes of
vector spaces over $\k$\,. $A$ is a unital
associative algebra over $\k$ and $C^{\bul}(A,A)$ is
the normalized Hochschild cochain complex of $A$ with
coefficients in $A$
\begin{equation}
\label{norm-Hoch}
C^{\bul}(A,A) = \hom((A/\k)^{\otimes \, \bul}\,,\, A)\,.
\end{equation}
The abbreviation SMC stands for ``symmetric monoidal category''
and the notation $\bu$ is reserved for the unit of a symmetric monoidal
category.
We also use the abbreviation SC for ``Swiss Cheese type'' when we discuss
the Swiss Cheese type symmetric operads, 2-operads, sets, ordinals,
and 2-trees.

\subsection{Organization and layout of the paper} All arguments of the paper can be restricted onto the setting
when we only care about operations on $C^\bullet(A,A)$ (and not on the pair $(C^\bullet(A,A),A)$.
Throughout  the paper we use terms like 'non SC part' or 'cochain part' to indicate that we restrict
to $C^\bullet(A,A)$ only.  The exposition is organized so that most of the constructions are first
introduced in the non SC  setting and then extended to the whole SC picture.  As a rule, this
SC extension is rather straightforward.  In our exposition we tried  to isolate the spots dealing with
the SC setting; we hope that the reader interested in proving Deligne's conjecture only
will be able to easily recognize these spots and drop them without any harm to understanding.

Let us now go over the content of the paper.
We start (Sec \ref{operad-La}) with defining an operad  $\op$ of natural operations on the infinite  collection
of  objects
\begin{equation}\label{collect1}
C^n(A,A),n=0,1,2,\ldots;\; A. 
\end{equation}

Next, we  explain how, using the functors of polysimplicial/cosimplicial totalization
(which are called {\em condensation}  in \cite{lattice}), we can convert
the operad $\op$ into a dg operad $|\op|$ 
 which acts on the pair of complexes:
$C^\bullet(A,A)$ and $A$. The operad $|\op|$ is the same as the operad $\La$ in Theorem \ref{main}.
 
In Sec \ref{combop} we give a combinatorial description
of $\op$ in terms of trees and  then reformulate it in  terms of sequences.  The latter descirption
is used in the rest of the paper.

Next, we invoke Batanin's 2-operad theory: in Sec \ref{review} we review the basic notions
of the theory and in Sec \ref{screview} we discuss an SC version of these notions (also due
to Batanin). This section is not needed for the cochain (Deilgne's) part of  the  SC conjecture.
In Section \ref{seq} we define a 2-operadic version $\seq$ of the operad $\op$. 

In Section \ref{condens} we apply the totalization (=condensation) procedure to
the operads $\seq$ and $\se$.  As a result we get
an operad $|\se|$ acting on the complex $C^\bullet(A,A)$ as well as its 2-operadic version
$|\seq|$.   At this moment the advantage of the 2-operadic approach can be seen: the 2-operad
$|\seq|$ turns out to be contractible, contrary to $|\se|$.  

We conclude the section with extending the above mentioned constructions to the SC setting.
We obtain  a contractible SC 2-operad $|\scseq|$ which acts on
the pair $(C^\bullet(A,A),A)$ .  If this operad  satisfied
a technical condition of being {\em reduced},  Batanin's theory  would imply an action of Voronov's
SC operad on $(C^\bullet(A,A),A)$. But $|\scseq|$ happens to be  non-reduced which causes
us to find a reduced contractible sub-operad $\scR$ of $|\scseq|$, see Sec \ref{braces}.
Using a similar approach we also construct a suboperad  $\scoR$ of $|\se|$.  The action of the
operad $|\scoR|$ on $C^\bullet(A,A)$  seems to be equivalent to the celebrated {\em brace structure}
on $C^\bullet(A,A)$  (\cite{GerVor}, \cite{Ezra}, \cite{Gruzia}).
  Batanin's theory can now be applied to
$|\scR|$;  we get  an action on $(C^\bullet(A,A),A)$ of a certain operad $E$ which is homotopy equivalent
to Voronov's SC operad (the operad $E$ is the symmetrization of a cofibrant resolution of $\scR$,  i.e.
 $E:=\sym\ \R \scR$,  see (\ref{Glavnoe})). 

 It also follows that this action passes through the action
of $\scoR$ that is we have a map of operads $E\to \scoR$. We prove that this map is a weak equivalence,
see Theorem \ref{SymRbr-braces};  the  proof of this theorem occupies the whole
Sec \ref{proof1}.  We are now ready for proving the SC conjecture (Sec. \ref{proof2}).
There  is an Appendix  which  contains a certain  contractibility statement needed for proving
Lemma \ref{lmfil1}.

\noindent
{\bf Acknowledgment.}
We would like to thank M. Batanin and J. Bergner for useful
discussions. We also thank anonymous referees for carefully 
reading the paper and many useful remarks and suggestions. 
A big part of this work was done when V.D.
was a Boas Assistant Professor of Mathematics Department
at Northwestern University. During these two years V.D. benefited
from working at Northwestern so much that he
feels as if he finished one more graduate school.
V.D. cordially thanks Mathematics
Department at Northwestern University for this time.
The results of this work were presented at the famous
Sullivan's Einstein Chair Seminar. We would like to thank
the participants of this seminar for questions and
useful comments. We especially thank D. Sullivan
for his remarks which motivated us to
rewrite the formulation of our main result in this paper.
D.T. and B.T. are supported by
NSF grants. V.D. is supported by the 
NSF grant DMS 0856196, Regent's Faculty Fellowship, and the Grant for Support
of Scientific Schools NSh-8065.2006.2. A part of this work was done when V.D. 
lived in Irvine and participated in the vanpool program to commute 
to the UCR campus. V.D. would like to thank Transportation and Parking 
Services of the UC Riverside for their work.

\section{The  operad $\op$ of natural operations on
 the objects $C^{n}(A,A),n=0,1,2,\ldots; A)$}
 \label{operad-La}

Let
 $A$ be a unital monoid in some tensor
(not necessarily symmetric) category (for example, in the category of complexes over a field).
 Consider the full nonsymmetric endomorphism operad
of $A$
$$
C_A(n):= \hom(A^{\otimes n};A), \qquad
n\geq 0.
$$

It is clear that $A$ is naturally a
$C_A$-algebra. The associative  unital
structure on $A$ gives rise to a map of
nonsymmetric operads $\as\to C_A$, where $\as$ is the
nonsymmetric operad of sets controlling unital monoids;
each space $\as(n)$, $n\geq 0$, is a point.

We fix a set of colors
$\chrom:=\nat\sqcup \{\a\}$ and define a
 $\chrom$-colored symmetric operad $\op$ in
the category of sets as an operad
 whose algebra structure on an $\chrom$-family
 of objects $(C(n), n\in \nat; A)$ is:

 --- a nonsymmetric operad structure on the collection
 of objects $C(n)$;

 --- a map of nonsymmetric operads $\as\to C$;

 --- a $C$-algebra structure on $A$.

The operad $\op$ has the following sets of operations:

 $$\text{---}\quad \op(k)_{n_1,n_2,\ldots,n_k}^n:=
\op((n_1,n_2,\ldots,n_k)\mapsto n))
 $$
 where all  the entries are in $\nat$;

 $$\text{---}\quad\op(k,N)_{n_1,n_2,\ldots,n_k}:=
\op((n_1,n_2,\ldots,n_k,\underbrace{\a,\a,\ldots\a}\limits_N)\mapsto
\a)\,, \qquad N \ge 0\,.
 $$

The operadic sets for other colorings are
empty.

The sets  $ \op(k)_{n_1,n_2,\ldots,n_k}^n$  form a $\nat$-colored operad in the obvious way.  Call this operad
{\em the cochain part of $\op$}. An algebra over this operad is a non-symmetric 
operad $C$ equipped with a map (of non-symmetric operads)
$$
\as \to C\,.
$$

Later on  (see \ref{planar})  an explicit combinatorial description of the operad $\op$ will be given.

\subsubsection{}\label{pocos} The unary operations in
the colored operad $\op$ endow the set of colors with
the following category structure:

--- $\hom(n, \a) = \emptyset$ for all
$n > 0$, $\hom(0,\a)$  is a one-point set;

--- $\hom(\a,n) = \emptyset$ for all
$n\in \nat$;

--- $\hom(n,m)=\hom_\Delta([n],[m])$ for all
$n,m\in \nat$;

--- $\hom(\a,\a)=\{\Id\}$.

This implies that the operadic sets of our
colored operad $\op$ have a natural
polysimplicial/cosimplicial structure,
namely:

the collection of sets
$$
\op(k)_{n_1,n_2,\ldots,n_k}^n,
$$
as $n_i,n$ run through $\nat$, is a functor
$$
\op(k):(\Delta^\opp)^k\times\Delta\to \Sets,
$$
(the functor is simplicial in each of the
arguments $n_1,n_2,\ldots,n_k$ and
cosimplicial in $n$);

likewise, for each $N$, the collection of
sets
$$
\op(k,N)_{n_1,n_2,\ldots,n_k}
$$
forms a functor $\op(k,N):(\Delta^\opp)^k\to
\Sets$.

\subsubsection{}
\label{202}
Let $\S$ be a
cosimplicial complex given by
$$
\S([n])^\bullet:=\overline{C_{-\bullet}}(\Delta^n,\k),
$$
where the complex on the right hand side is the normalized
chain complex of the simplex $\Delta^n$ put
in the non-positive degrees.

 Using this
complex, we can convert
polysimplicial/cosimplicial sets into
complexes.

Namely, let $F:(\Delta^{\opp})^k\to\Sets$ be
a functor.  Set
$$
|F|:=\k[F]\otimes_{(\Delta^\opp)^k}
\S^{\boxtimes k},
$$
where $S^{\boxtimes k}:\Delta^k\to
\complexes:$
$$
\S^{\boxtimes k}([n_1],[n_2],\ldots,[n_k]):=
\bigotimes\limits_{i=1}^k \S([n_i]).
$$

Given a functor
$$
G:(\Delta^\opp)^k\times \Delta\to \Sets,
$$
denote by $G^n$ the
evaluation at $[n]\in \Delta$ so that
$$
G^n:(\Delta^\opp)^k \to \Sets
$$
and
$$
n\mapsto G^n
$$
is a functor from $\Delta$ to the category
of $k$-simplicial sets. Set
$$
|G|:=\hom_{\Delta}(\S^\bullet,|G^\bullet|).
$$
\subsubsection{} Set
$$
|\op|(k):=|\op(k)|;
$$
$$
|\op|(k,N):=|\op(k,N)|.
$$

We see that these spaces form a 2-colored
DG operad. Denote this two-colored operad by
$|\op|$.

Now let $A$ be a unital associative algebra
over the field $\k$\,. It is easy to see that the
normalized Hochschild cochain complex $C^{\bul}(A,A)$
(\ref{norm-Hoch}) can be written as
$$
C^\bullet(A,A):=\hom_\Delta(S^*,C_A(*))\,.
$$
Therefore the DG operad $|\op|$ acts on the
pair $(C^\bullet(A,A), A)$\,. This two-colored DG operad
$|\op|$ is the desired operad $\La$ of natural operations
on the pair $(C^\bullet(A,A), A)$ and our Theorem \ref{main}
can be reformulated as
\begin{Theorem}
\label{th}
The operad $|\op|$ is weakly
equivalent to the singular chain operad of Voronov's
Swiss Cheese operad $\SC_2$. The induced
action of the homology operad $H_{-\bul}(\SC_2)$
on the pair $(HH^{\bul}(A), A)$ recovers the one
from Proposition \ref{HH-swiss}\,.
\end{Theorem}
We prove this theorem in Section \ref{pfth}.

\noindent\\
{\bf Remark.} Our method also works in the
topological setting: one can apply the
topological realization functors to the
polysimplicial/cosimplicial sets from
\ref{pocos} so as to get a topological
colored operad $|\op|_{\top}$. This
operad can be proven to be weakly equivalent
to Voronov's Swiss Cheese operad.

\section{Two combinatorial descriptions of the operad $\op$}\label{combop}
Our first  description will be in terms of planar trees. Next, we will explain a
transition from the tree description to another one, in terms of sequences.

Each construction will be first introduced for the cochain part of $\op$
 and then 
extended to the whole operad.    These  extensions for both constructions are rather
straightforward.

We will start with fixing a more convenient language.
\subsection{Finite ordinals instead of natural numbers}   Recall that our set
of colors is $\nat\sqcup \{\a\}$, and that
an $\op$-algebra structure on the collection
of spaces $(C(n), n\in \nat; A)$ is the same
as a nonsymmetric operad structure on the
collection of spaces $C(n), n\in \nat$, a
map of operads $\as\to C$, and a $C$-algebra
structure on $A$. The definition of a
nonsymmetric operad implies that we have a
total order on the set of arguments so that
it is better to replace the natural numbers
with isomorphism classes of finite ordinals:
 the number $n$ gets replaced with the
ordinal $<n>=\{1<2<\cdots<n\}$.

Given finite sets $S$, $S_{\c}$, an $S$-family
$\{I_s\}_{s\in S}$ and an $S_{\c}$-family
$\{I_s\}_{s\in S_{\c}}$
of finite (possibly empty) ordinals,
an ordinal $J$\,, and a set $S_\a$, we
then have the following operadic sets:
\begin{equation}
\label{dano1}
\op(S)_{\{I_s\}_{s\in S}}^J;
\end{equation}
\begin{equation}
\label{dano2}
\op(S_\c,S_\a)_{\{I_s\}_{s\in S_\c}},
\end{equation}
where in (\ref{dano1}) the set of arguments
is $S$ and the coloring of $s\in S$ is
$I_s$, the result has the color $J$.
In (\ref{dano2}), the set of arguments is
$S_{\c} \sqcup S_\a$ the argument $s\in S_{\c}$ has
color $I_s$ and all arguments from $S_\a$
have color $\a$. The result also has color
$\a$.

\subsection{Planar trees}\label{planar}
\subsubsection{ The cochain part of $\op$ via planar trees}

For a finite  set $S$ and ordinals
$I_s,s\in S$; $J$, we describe
$$
\op(S)_{\{I_s\}_{s\in S}}^J
$$
as the set of equivalence classes
of planar trees $T$ with the following
structure:

--- a subset of the set of vertices of a tree $T$
is identified with $S \sqcup J$ in such a way that
with elements of $J$ we may only identify the
terminal vertices of $T$. We call the
vertices identified with elements of $S \sqcup J$
marked.

--- the ordered set of edges originating at the
vertex marked by $s\in S$ is identified with
$I_s$\,.

Notice that, the subset of
vertices identified with $J$
acquires from $J$ a natural linear order.
We require that this linear order coincides with
the order which is obtained by going around the
tree in the clockwise direction starting from the
root vertex.

The equivalence relation is
the finest one in which two such trees are
equivalent if one of them can be obtained
from the other by either:

 the contraction of
an edge  with unmarked ends

or: removing an unmarked vertex with only
one edge originating from it and joining the
two edges adjacent to this vertex into one
edge.

~\\
{\bf Example.} The planar tree
$T$ in figure \ref{f} represents an element
in $\op(S)_{\{I_s\}_{s\in S}}^J$ with
$S=\{s_1, s_2\}$\,, $J=\{j_1, j_2, j_3\}$\,,
$I_{s_1}= \emptyset$\,, and $I_{s_2} = <3>$\,.
\begin{figure}[htb]
\begin{center}
\includegraphics[width=5cm,height=5cm]{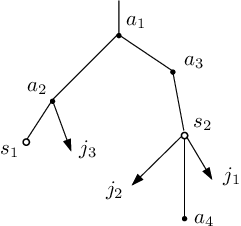}
\caption{Tree $T$} \label{f}
\end{center}
\end{figure}
In all the figures we
use circles to denote the vertices marked
by elements of $S$ and arrows to denote vertices
marked by elements of $J$\,. Thus, in figure \ref{f}
the vertices $a_1$, $a_2$, $a_3$, and $a_4$ are
unmarked. The vertices $a_1$ and $a_2$ correspond
to the product in $\as(2)$, $a_3$ corresponds to
the identity operation in $\as(1)$, and $a_4$
corresponds to the unit in $\as(0)$\,.

In figures \ref{f1} and \ref{f2} we depict
the trees $T_1$ and $T_2$ which are equivalent
to the original tree $T$.
\begin{figure}[htb]
\includegraphics[width=5cm,height=5cm]{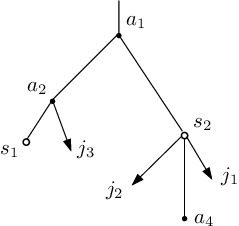}
\caption{Tree $T_1$} \label{f1}
\end{figure}
\begin{figure}[htb]
\includegraphics[width=5.25cm,height=5.25cm]{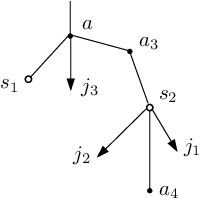}
\caption{Tree $T_2$} \label{f2}
\end{figure}
The tree $T_1$ is obtained from $T$ by
removing the unmarked vertex $a_3$ and joining the
two edges adjacent to this vertex into one
edge. The tree $T_2$ is obtained from $T$
by contracting the edge with the
unmarked ends $a_1$ and $a_2$\,.
The unmarked vertex $a$ of the tree $T_2$
(figure \ref{f2}) corresponds to the unique
element of $\as(3)$\,.

Applying both of the equivalence operations
to the tree $T$ in figure \ref{f} we obtain the
tree $T_3$ depicted in figure \ref{f3}.
\begin{figure}[htb]
\includegraphics[width=6cm,height=5cm]{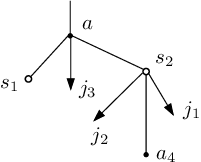}
\caption{Tree $T_3$} \label{f3}
\end{figure}
Although the tree $T_3$ has unmarked vertices $a$
and $a_4$, it is no longer possible to apply any
equivalence operation to $T_3$\,.
We call such trees {\it minimal}.
It is obvious that every equivalence class
of $\op(S)_{\{I_s\}_{s\in S}}^J$ contains
at least one minimal tree.

The equivalence class containing all these planar trees
$T$, $T_1$, $T_2$, and $T_3$ corresponds to the operation
which sends a Hochschild cochain $P_1\in C_A(0)$ and a Hochschild
cochain $P_2\in C_A(3)$ to the Hochschild cochain
$Q\in C_A(3)$ defined by the formula
$$
Q (b_1, b_2, b_3) = P_2(b_1, 1, b_2)\, b_3 \, P_1\,,
$$
$$
b_1, b_2, b_3 \in A\,.
$$

\subsubsection{The whole operad $\op$ in terms of planar trees} Let us now describe the set
$$
\op(S_\c,S_\a)_{\{I_s\}_{s\in S_\c}},
$$
where we use the same notation as above.

Each element of this set can be represented
by a planar  tree $T$
 with the following additional structure:

--- a subset of the set of vertices of $T$
is identified with $S_{\c}\sqcup S_{\a}$ in such
a way that
with elements of $S_{\a}$ we may only identify the
terminal vertices of $T$\,. We call the vertices
identified with elements of $S_{\c}\sqcup S_{\a}$
marked;

--- the  ordered set of edges originating at the
vertex marked by $s\in S_{\c}$ is identified with
$I_s$\,.

The equivalence relation on the set of
isomorphism classes  of such trees is
defined in the same way as in the previous
section.

This description implies the following
identification:
$$
\op(S_\c,S_\a)_{\{I_s\}_{s\in
S_\c}}=\bigsqcup_{>\in \ord(S_\a)}
\op(S_\c)^{S_\a,>}_{\{I_s\}_{s\in S_\c}},
$$
where $\ord(S_{\a})$ is the set of all total
orders on $S_{\a}$.

Let us also describe the degenerate cases.
In the case $S$ is the empty set $\emptyset$ we have
$$
\op (\emptyset)^J = \as(J)\,.
$$
If $S_{\a} = \emptyset$ then
$$
\op(S_\c, \emptyset)_{\{I_s\}_{s\in
S_\c}}= \op(S_\c)^{\emptyset}_{\{I_s\}_{s\in S_\c}}\,.
$$
Finally, if $S_{\c}$ is empty then
$$
\op(\emptyset , S_\a) = \bigsqcup_{>\in \ord(S_\a)}
\as(S_\a,>)\,.
$$

\subsection{Replacing trees with sequences}
\label{posle} 
We will put into a correspondence to any planar  tree 
from the previous subsection a certain sequence which will lead to another
description of $\op$.  We start with the cochain part of $\op$.

\subsubsection{Cochain part of $\op$ in terms of sequences, I}

We need the following notation. Given a
vertex $v$ of a planar tree marked by
an element $s\in S$, let us draw a
little circle centered at this vertex. This
circle gets split into sectors, the set of
these sectors is totally ordered in the
clockwise order. Denote this ordered set by
$I'_s$. The set of edges originating at $v$
is naturally identified with $\vec{I'_s}$,
where $\vec{I'_s}$ is the set of pairs
$\vec{i_1i_2}$, where $i_2$ is an immediate
successor of $i_1$ and $i_1,i_2\in I'_s$. We
see that $I'_s$ is the next ordinal after
$I_s$. Below, given an ordinal $K$, we
denote by $K'$ its next ordinal.

 Given a planar tree $T$
which defines an element  $\overline{T}\in
\op(S)_{\{I_s\}_{s\in S}}^J$, let us
consider its small tubular neighborhood and
let us walk along its boundary starting from
the root vertex of our tree in the
clockwise direction. On our way, we will
meet the vertices marked by elements of $S$
and vertices marked by elements of $J$.
(The latter ones are terminal according to
our requirement.) Every time we approach a
vertex $v$ marked by $s\in S$,
we are at a certain sector from
$I'_s$.  Thus, given a planar tree $T$
representing an element $\overline{T}\in
\op(S)_{\{I_s\}_{s\in S}}^J$, we obtain a
total order $>_T$ on the set
$$
\bigsqcup_{s\in S} I'_s\sqcup J\,.
$$
{\bf Example.} Let us show how we obtain
the order for the tree $T_3$ given in
figure \ref{f3}. This tree represents
an element in
$\op(\{s_1,s_2\})^{\{ j_1, j_2, j_3 \} }_{I_{s_1}, I_{s_2}}$
where $I_{s_1}$ is empty and $I_{s_2}=<3>$\,.
This means that the vertex labeled by $s_1$
(see figure \ref{poryadok})
is surrounded by a single sector $s_1^1$, while
the vertex labeled by $s_2$ is surrounded by
four sectors $s_2^1$, $s_2^2$, $s_2^3$,
$s_2^4$ which we number in the clockwise
direction.
\begin{figure}[htb]
\begin{center}
\includegraphics[width=8cm,height=6cm]{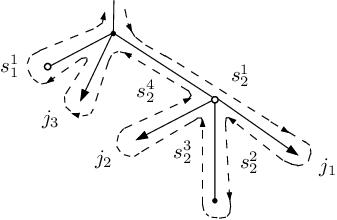}
\caption{Tree $T_3$} \label{poryadok}
\end{center}
\end{figure}
Walking along the boundary of
a small tubular neighborhood of $T_3$,
as it is shown on figure \ref{poryadok},
we get the following order on the set
$\{s^1_1,  s_2^1, s_2^2, s_2^3,
s_2^4, j_1, j_2, j_3 \}$ :
$$
s_2^1 < j_1 < s_2^2 < s_2^3 < j_2 < s_2^4 < j_3 < s_1^1\,.
$$

For every planar tree $T$ representing an
element $\overline{T}\in
\op(S)_{\{I_s\}_{s\in S}}^J$ the
corresponding  total order $>_T$ satisfies:

1) let $T_1, T_2$ be planar trees
representing the same element
$$
\overline{T_1}=\overline{T_2} \in
\op(S)_{\{I_s\}_{s\in S}}^J
$$
Then $>_{T_1}=>_{T_2}$. Hence, for
each $\overline{T}\in\op(S)_{\{I_s\}_{s\in
S}}^J$ we have a well defined order, to be
denoted by $>_{\overline T}$;

2) The total order $>_T=>_{\overline{T}}$
agrees with the existing orders on $I'_s,J$;

3) given distinct $s_1,s_2\in S$ it is
impossible to find $i_1,j_1\in I'_{s_1}$;
$i_2,j_2\in I'_{s_2}$ such that
$$
i_1 <_T i_2 <_T j_1 <_T j_2\,.
$$

Let us denote the set of all total orders satisfying
conditions 2) and 3) by $\Ord(S)^J_{\{I_s\}_{s\in S}}$\,.

\begin{Proposition}
The set $\Ord(S)^J_{\{I_s\}_{s\in S}}$
is in 1-to-1 correspondence with
$\op(S)_{\{I_s\}_{s\in S}}^J$.
\end{Proposition}

\begin{proof} Let us describe an inductive
construction which assigns to each
total order $\pi$ on
\begin{equation}
\label{stroka}
\bigsqcup_{s\in S} I'_s\sqcup J
\end{equation}
satisfying conditions 2) and 3)
a minimal tree $T$ which recovers the
order $\pi$ by walking along a small
tubular neighborhood of $T$\,.

The induction goes by the order $|S|$
of the set $S$\,.

For $S = \emptyset$ the set
$\Ord(S)^J_{\{I_s\}_{s\in S}}$ consists of
a single element. That is the given order
on $J$\,. In this case it is very easy to
find a minimal tree which recovers this
order. It is also easy to see that such a tree
is unique.

Let us suppose that we can construct a
desired minimal tree for all elements of
$\Ord(S_0)^J_{\{I_s\}_{s\in S_0}}$ if
$|S_0|< |S|$\,. We need to present a
construction for every
$\pi \in \Ord(S)^J_{\{I_s\}_{s\in S}}$\,.

Condition 3) implies that for an arbitrary
pair $s, \ts \in S$ exactly one of the
following options realizes:
\begin{enumerate}
\item all elements of $I'_{\ts}$ are smaller
than elements of $I'_s$\,,

\item all elements of $I'_{\ts}$ are greater
than elements of $I'_s$\,,

\item $I'_{\ts}$ splits into two non-empty
subsets such that all elements of the first
subset are smaller than all elements of $I'_s$
while all the elements of the second subset
are greater than elements of $I'_s$\,

\item same as  (3) with $s$ and $\ts$ interchanged.

\end{enumerate}
If the third (resp. fourth) option realizes we say that
$s < \ts$ (resp. $\ts<s$\,). Thus we get a partial order
on the set $S$\,.

Since $S$ is finite, it has at least
one minimal element. Let us denote this
element by $s_{min}$ and introduce
the interval $\tI_{s_{min}}$
of the ordinal (\ref{stroka}) between
the minimal element of $I'_{s_{min}}$ and the maximal
element of $I'_{s_{min}}$\,. It is obvious that
$\tI_{s_{min}}$ consists of elements of $I'_{s_{min}}$
and some elements of $J$\,.

Let us consider the set
\begin{equation}
\label{stroka1}
\bigsqcup_{s\in S^{(1)} }
I'_s\sqcup J^{(1)}\,,
\end{equation}
where $S^{(1)}= S\setminus \{ s_{min}\} $ and
$J^{(1)}$ is obtained from $J$ by attaching
the element $s_{min}$ and removing
those elements of $J$ which belong to
the interval $\tI_{s_{min}}$\,.
In other words,
\begin{equation}
\label{J-1}
J^{(1)} = J \sqcup \{ s_{min} \}
\setminus ( J \cap \tI_{s_{min}} ) \,.
\end{equation}

Notice that, the set (\ref{stroka1})
is obtained from (\ref{stroka}) by
replacing the interval $\tI_{s_{min}}$
by a single element $s_{min}$\,.
Hence, (\ref{stroka1})
acquires a natural total order. Let us denote this
order by $\pi^{(1)}$\,.

It is not hard to see
that $\pi^{(1)}$ satisfies conditions 2) and 3)
and hence is an element of the set
$$
\Ord(S^{(1)})^{J^{(1)}}_{\{I_s\}_{s\in S^{(1)} }}\,.
$$

Since $|S^{(1)}| < |S|$ we can assign to $\pi^{(1)}$
a minimal tree $T^{(1)}$ which recovers
the order $\pi^{(1)}$ on (\ref{stroka1})\,.

To construct the desired tree $T$ we observe
that the element $s_{min}$ is identified with
an external vertex $v$ of $T^{(1)}$\,.
So, we draw from
this vertex $v$ edges labeled by elements
$\vec{i_1 i_2}$ of $\vec{I'}_{s_{min}}$\,.
Recall that $\vec{I'}_{s_{min}}$ consists of pairs
$\vec{i_1 i_2}$, where $i_2$ is an immediate
successor of $i_1$ and $i_1,i_2\in I'_{s_{min}}$\,.

Let us denote by $t^{i_1 i_2}$ the terminal
vertex of the edge corresponding to
$\vec{i_1 i_2}$\,.

If there are no elements of $J$ between $i_1$
and $i_2$ then we leave $t^{i_1 i_2}$
as an unmarked terminal vertex of
the tree $T$.

If there is only one element $j$ of $J$
between $i_1$ and $i_2$ we leave
$t^{i_1 i_2}$ as a terminal vertex of $T$
and mark it by $j$\,.

Finally, if we have elements
$j_1, \dots, j_m \in J$ ($m>1$)
between $i_1$ and $i_2$, then we draw from
the vertex $t^{i_1 i_2}$ exactly $m$
terminal edges. We leave $t^{i_1 i_2}$
unmarked and mark the corresponding
terminal vertices by $j_1, \dots, j_m$ in
the clockwise direction.

Let us denote the resulting tree by $T$\,.
It is not hard to see that, since
$T^{(1)}$ recovers the order $\pi^{(1)}$
on (\ref{stroka1})
the tree $T$ recovers the order $\pi$ on
(\ref{stroka}). It is also
obvious that, since the tree $T^{(1)}$ is minimal,
so is $T$\,.

We already have a map from the set
$\op(S)^J_{\{I_s\}_{s\in S}}$ to
the set $\Ord(S)^J_{\{I_s\}_{s\in S}}$
which is defined by assigning the total
order to a tree. Let us denote this
map by $\nu_{\ord}$
$$
\nu_{\ord} : \op(S)^J_{\{I_s\}_{s\in S}}
\to
\Ord(S)^J_{\{I_s\}_{s\in S}}\,.
$$
The above construction provides us
with the map in the opposite direction:
$$
\nu_{\rm tree} : \Ord(S)^J_{\{I_s\}_{s\in S}}
\to
\op(S)^J_{\{I_s\}_{s\in S}}\,.
$$
It is clear from the construction that
the composition $\nu_{\ord}\circ \nu_{\rm tree}$
is the identity on $\Ord(S)^J_{\{I_s\}_{s\in S}}$\,.

It is not hard to verify that if we start
with a minimal tree $T$ representing an
element $\overline{T} \in \op(S)^J_{\{I_s\}_{s\in S}}$\,,
and assign to $T$ the total order $\pi$ from
$\Ord(S)^J_{\{I_s\}_{s\in S}}$\,, then the
above construction gives us back exactly
the same minimal tree $T$\,.
This implies that the composition
$\nu_{\rm tree} \circ \nu_{\ord}$ is the
identity on the set $\op(S)^J_{\{I_s\}_{s\in S}}$
and the proposition
follows\footnote{In particular, it implies that
in each equivalence class of trees
there is exactly one minimal tree.}.
\end{proof}

\noindent\\
{\bf Remark.} The construction presented
in the proof
is reminiscent of Kontsevich-Soibelman
pairs of complementary orders \cite{KS}.

\subsubsection{Cochain part of $\op$ via sequences, II: modification}\label{modif}  Given a total
order as above, we can construct a map

$$
Q:\bigsqcup\limits_{s\in S} I'_s\to J'
$$
as follows. We identify $J'=J\sqcup \{M\}$,
where $M>J$. Set $Q(x)=j$ if $j$ is the
minimal element from $J$ such that $j>x$; if
there is no such $j$, set $Q(x)=M$.

Thus, given a total order as in the previous
subsection, we obtain the following data:

--- a total order on the set
 $$\I:=\bigsqcup\limits_{s\in
S} I'_s;
$$
a non-decreasing map
$$
\I\to J'.
$$

 These data should satisfy:

i) the order  on $\I$ agrees with those on
each $I'_s$;

ii) same as condition 3) from Sec
\ref{posle}.

Denote the set of such objects by
$$
\se(S)_{\{I'_s\}_{s\in S}}^{J'}\,.
$$

This set is in 1-to-1 correspondence with
the set of total orders from the previous
subsection, hence we have a bijection
with the set $\OO(S)_{\{I_s\}_{s\in S}}^J$:
\begin{equation}
\label{seqtree1}
\se(S)_{\{I'_s\}_{s\in S}}^{J'}\to
\OO(S)_{\{I_s\}_{s\in S}}^J.
\end{equation}

\subsection{The whole operad $\op$ in terms of sequences} Likewise, one identifies
the set $\OO(S_\c,S_\a)_{\{I_s\}_{s\in
S_\c}}$ with the set of total orders on
$$
\bigsqcup\limits_{s\in S_\c} I'_s\sqcup S_\a
$$
satisfying:

--- the total order agrees with those on
each $I'_s$;

--- same as condition 3) from Sec.
\ref{posle}.

Denote the set of such total orders by
$\se(S_\c,S_\a)_{\{I_s\}_{s\in S_\c}}$.

The construction of the 1-to-1
correspondence
\begin{equation}\label{seqtree2}
\se(S_\c,S_\a)_{\{I'_s\}_{s\in S_\c}}\to
\OO(S_\c,S_\a)_{\{I_s\}_{s\in S_\c}}
\end{equation}
is the same as in the previous subsection.

\subsection{Operadic structure on $\se$}\label{compse} Let $\nat'$ be the set of
isomorphism classes of non-empty finite
ordinals. The identifications
(\ref{seqtree1}), (\ref{seqtree2}) imply
that the colored operad structure on $\OO$
induces a colored operad structure on the
collection of spaces $\se$. It turns out
that this operadic structure can be
naturally formulated in terms of $\se$. 

~\\

\noindent
{\bf Warning.}  We will not use the symbol $'$
anymore when talking about ordinals from
$\nat'$. The reason is that in the sequel,
instead of the operad $\OO$, the isomorphic
operad $\se$ will be used.

\subsubsection{Operadic structure on the cochain part of $\se$} Let $T$ be a finite set and
let $S_t$ be a $T$-family of finite sets.
Let $S:=\sqcup_t S_t$ and $p:S\to T$ be the
map which sends  $S_t$ to $t$.

Suppose we are given ordinals $I_s$,
$s\in S$; $J_t,t\in T$, and $J$.

Describe the operadic composition
$$
\se(T)_{\{J_t\}_{t\in T}}^J\times
\prod\limits_{t\in T}
\se(S_t)^{J_t}_{\{I_s\}_{s\in S_t}}\to
\se(S)_{\{I_s\}_{s\in S}}^J\,.
$$

Let $u\in \se(T)_{\{J_t\}_{t\in T}}^J$ and
$u_t\in \se(S_t)^{J_t}_{\{I_s\}_{s\in
S_t}}$. Let us describe  the composition $v$
of these elements.

1) the total order $>_v$ is defined as a
unique one

---  which agrees with the orders $>_{u_t}$ on
$$
\sqcup_{s\in S_t} I_s\subset \sqcup_{s\in S}
I_s
$$
for each $t\in T$;

--- for which the map
\begin{equation}\label{seqq}
\sqcup_{t\in T} F_{u_t}:\sqcup_{s\in S}
I_s\to (\sqcup_{t\in T} J_t,>_u)
\end{equation}
is non-decreasing.

2) the map $F_v$ is just the composition  of
(\ref{seqq}) with the map $F_u$.

\noindent\\
{\bf Remark.} The non-SC part of the operad $\se$ as well as 
its totalization was considered in earlier papers on 
Deligne's conjecture and its variations. 
Thus, 
the non-SC part of $\se$ is isomorphic to the 
second filtration stage of the lattice path operad 
introduced by M. Batanin and C. Berger in \cite{lattice}. 
The non-SC part of the totalization of $\se$ was considered 
in papers \cite{M-SmithJAMS} and \cite{M-SmithAJM} by
J. E. McClure and J. H. Smith.

\subsubsection{Operadic structure on the whole operad $\se$} To describe the
remaining composition maps we consider
sets $S_\c, S_\a, T_\c ,T_\a$
and let $P:S_\c\sqcup S_\a\to T_\c\sqcup
T_\a$ be a map such that $P^{-1}T_\c\subset
S_\c$. For $t\in T_\a$ we set
$(P^{-1}t)_\a:=P^{-1}t\cap S_\a$;
$(P^{-1}t)_\c:=P^{-1}t\cap S_\c$.

Let $\{I_s\}_{s\in S_\c}$; $\{J_t\}_{t\in
T_\c}$;  be  non-empty ordinals. We need to define
the following composition map:
\begin{multline*}
\se(T_\c,T_\a)_{\{J_t\}_{t\in T_\c}}\times
\prod\limits_{t\in
T_\c}\se(P^{-1}t)^{J_t}_{\{I_s\}_{s\in
P^{-1}t}}\times \prod\limits_{t\in
T_\a}\se((P^{-1}t)_\c,(P^{-1}t)_\a)_{\{I_s\}_{s\in
(P^{-1}t)_\c}}\\
\to \se(S_\c,S_\a)_{\{I_s\}_{s\in S_\c}}\,.
\end{multline*}

Choose elements
$$
v\in \se(T_\c,T_\a)_{\{J_t\}_{t\in T_\c}};
$$
$$
u_t\in \se(P^{-1}t)^{J_t}_{\{I_s\}_{s\in
P^{-1}t}};\ t\in T_\c\,,
$$
$$
u_t\in
\se((P^{-1}t)_\c,(P^{-1}t)_\a)_{\{I_s\}_{s\in
(P^{-1}t)_\c}};\ t\in T_\a
$$
and denote their composition by $w$.

Let us set
$$
\I_w:=\bigsqcup\limits_{s\in S_\c} I_s
\sqcup S_\a.
$$
and define  a map
$$
F:\I_w\to \I_v,
$$
where
$$
\I_v=\bigsqcup\limits_{t\in T_\c}
J_t\sqcup T_\a,
$$

as follows:

--- If $t\in T_\c$ then the restriction
of $F$ to the subset
$$
\I_{u_t}:=\bigsqcup\limits_{s\in P^{-1}t}
I_s,
$$
should coincide with the map
$F_{u_t}:\I_{u_t}\to J_t$;

--- if $t\in T_\a$ then the restriction of $F$ to
the subset
$$
\I_{u_t}:=\bigsqcup\limits_{s\in
(P^{-1}t)_\c} I_s \sqcup (P^{-1}t)_\a
$$
should send every element to $t$.

We define the order $>_w$ as the unique one
for which the map $F$ is non-decreasing and
which agrees with the orders $>_{u_t}$ on
$\I_{u_t}$, $t\in T_\c\sqcup T_\a$.

\section{ Review of 2-operads}
\label{review} We are
going to remind the basic definitions from
Batanin's theory of 2-operads which will be
used below.  Next, we review Batanin's definition of
SC 2-operad. 

{\em An ordinal} is a  finite totally
ordered set. Another name for ordinals is a
1-tree.

{\em A 2-tree $\t$} is a pair of ordinals
$S,T$ along with an order preserving  map
$\t:S\to T$.

A 2-tree is called {\em pruned} if the map
$\t$ is surjective.

{\em A map of 2-trees} $$P:(\t:S\to T)\to
(\t_1:S_1\to T_1)$$ is a  pair of maps
$P_S:S\to S_1$; $P_T:T\to T_1$ such that
$\t_1P_S=P_T\t$; $P_T$ is order preserving;
$P_S$ preserves the order on each set
$\t^{-1}t$, $t\in T$.

This way, 2-trees form a category $\ttrees$.

Given  $s_1\in S_1$, we define a 2-tree
$P^{-1}s_1$ as follows:
$$
\t \, \Big|_{(P_S)^{-1} s_1} \, : \,
(P_S)^{-1} s_1\to (P_T)^{-1}\t_1(s_1).
$$

{\em A 2-operad} in a symmetric monoidal
category (SMC) $\C$ is defined as:

--- a functor $\cO:\ttrees^{\times}\to \C$,
where $\ttrees^{\times}$ is the groupoid of
isomorphisms of $\ttrees$ (note that every
object in this groupoid has the trivial
automorphism group);

--- for every map of 2-trees $P:\t\to \t_1$,
where $\t:S\to T$; $\t_1:S_1\to T_1$, there
should be given a map
$$
\cO(\t_1)\otimes\bigotimes\limits_{s_1\in
S_1} \cO(P^{-1}s_1)\to \cO(\t)
$$
called {\em the operadic composition map}.

These maps should satisfy a certain
associativity property. In order to
formulate it let us define the objects
$\cO(P)$, where $P:\t\to\t_1$ is a map of
2-trees as follows:
$$
\cO(P):=\bigotimes\limits_{s_1\in
S_1}\cO(P^{-1}s_1).
$$

The operadic insertion maps can be rewritten
as
$$
\cO(\t_1)\otimes \cO(P)\to \cO(\t).
$$

 Given a chain of
maps of 2-trees
$$
\t\stackrel P\to \t_1\stackrel Q\to \t_2,
$$
the operadic insertion maps naturally give
rise to a map
\begin{equation}
\label{prop}
\cO(Q)\otimes \cO(P)\to \cO(QP)\,.
\end{equation}
Indeed, for every $s_2\in S_2$, where
$t_2:S_2\to T_2$, the map $P$ naturally
restricts to a map of 2-trees
$$
P_{s_2}:(QP)^{-1}s_2\to Q^{-1}s_2
$$
and we have
$$
\cO(P)\cong
\bigotimes\limits_{s_2\in S_2} \cO(P_{s_2}).
$$

We then define the  map (\ref{prop}) as
follows:
$$
\cO(Q)\otimes \cO(P)\cong
\bigotimes\limits_{s_2\in S_2}
\cO(Q^{-1}s_2)\otimes \cO(P_{s_2})\to
\bigotimes\limits_{s_2\in S_2}
\cO((QP)^{-1}s_2)=\cO(QP).
$$
The associativity axiom requires that the
map (\ref{prop}) be associative: the
following  maps should coincide:
$$
\cO(R)\otimes \cO(Q)\otimes \cO(P)\to
\cO(RQ)\otimes \cO(P)\to \cO(RQP)
$$
and
$$
\cO(R)\otimes \cO(Q)\otimes \cO(P)\to
\cO(R)\otimes \cO(QP)\to \cO(RQP).
$$

\noindent\\
{\bf Remark.} 1-trees are simply ordinals and the definition
of a 1-operad based on 1-trees coincides with the definition
of a nonsymmetric operad.

\subsection{Colored 2-operads}
\subsubsection{Colored 2-trees}
 Fix a set of
colors $\chrom$.  Define a colored 2-tree
$\tau$ as:

--- a 2-tree $\t_\tau:S_\tau\to T_\tau$;

--- a map $\chi_\tau:S_\tau\to \chrom$;

--- an element $c_\tau\in\chrom$.
\subsubsection{} Given colored 2-trees
$\tau_1,\tau_2$ we define  their map
$P:\tau_1\to \tau_2$ as follows:

--- if $c_{\tau_1}=c_{\tau_2}$, then it is
just a map $P : \t_{\tau_1}\to
\t_{\tau_2}$ of the underlying 2-trees;

--- if $c_{\tau_1}\neq c_{\tau_2}$, then
we declare that there are no maps $\tau_1\to
\tau_2$.

 This way  colored 2-trees form a
category.

Given such a map and $s_2\in S_{\tau_2}$ the
2-tree $P^{-1}s_2$ naturally receives a
coloring as follows.

Recall that the 2-tree $P^{-1}s_2$ is
defined as
$$
\t_{\tau_1} \, \Big|_{(P_S)^{-1} s_2}\, : \, (P_S)^{-1} s_2 \to
(P_T)^{-1}\t_{\tau_2} s_2\,.
$$

We then define
$$
\chi_{P^{-1}s_2}:P_S^{-1}s_2\to \chrom
$$
as the restriction of $\chi_{\tau_1}$ and
set
$$
c_{P^{-1}s_2}:=\chi_{\tau_2}(s_2).
$$

\subsubsection{} We then define a colored
2-operad in a SMC $\C$ as:

--- a functor $\cO$ from the isomorphism groupoid
of the category of colored 2-trees to the
category $\C$;

--- for every map $P:\tau_1\to \tau_2$ of
colored 2-trees there should be given the
operadic composition map
$$
\cO(\tau_2)\otimes\bigotimes_{s_2\in
S_{\tau_2}} \cO(P^{-1}s_2)\to \cO(\tau_1).
$$

Next, given a map $P:\tau_1\to\tau_2$ we
define
$$
\cO(P):=\bigotimes\limits_{s_2\in
S_{\tau_2}} \cO(P^{-1}s_2)
$$
and observe that the operadic composition
maps naturally produce  maps
$$
\cO(Q)\otimes \cO(P)\to \cO(QP),
$$
where $P:\tau_1\to \tau_2$, $Q:\tau_2\to
\tau_3$.

Lastly we require the associativity of this
map in the same way as for the non-colored
2-operads.

\subsection{Unital colored 2-operads}
\subsubsection{} Given a color $c\in \chrom$,
consider a special 2-tree 
$
\u_c:\pt\to \pt
$
such that $\chi_{\u_c}$ sends $\pt$ to $c$
and $c_{\u_c}:=c$.

 For every isomorphism $P:\tau_1\to
\tau_2$ of  colored 2-trees every pre-image
$P^{-1}s_2$, $s_2\in S_{\tau_2}$, is
isomorphic (canonically) to $\u_c$, where
$c=\chi_{\tau_2}(s_2)$. Furthermore, for every colored
2-tree $\tau$ the pre-image $Q^{-1}\, \pt$ of the point for
a unique map $Q : \tau \to \u_{c_{\tau}}$ is equal to $\tau$\,.

Let $\bu$ be the unit of the underlying symmetric monoidal
category. Define {\em a unital 2-operad } as a colored
2-operad $\cO$ along with maps 
\begin{equation}\label{unitalmap}
\bu\to
\cO(\u_c)
\end{equation}
for each $c\in \chrom$ satisfying:

--- for every isomorphism $P:\tau_1 \to \tau_2$,
the map
$$
\cO(\tau_2)\cong \cO(\tau_2)\otimes {\bu
}^{\otimes S_{\tau_2}}\to \cO(\tau_2)\otimes
\bigotimes\limits_{s_2\in S_{\tau_2}}
\cO(P^{-1}s_2)\to \cO(\tau_1)
$$
coincides with the map $\cO(\tau_2)\to
\cO(\tau_1)$ induced by $P^{-1}$ from the
definition of  $\cO$ as a functor from the
isomorphism groupoid of the category of
colored 2-trees.

--- for every colored 2-tree $\tau$ the composition
$$
 \cO(\tau) \cong \bu \otimes \cO(\tau) \to
\cO (\u_{c_{\tau}}) \otimes \cO(\tau) \to \cO(\tau)
$$
is the identity on $\cO(\tau)$\,.

\subsubsection{Pruned colored 2-operads}
\label{pruned}
 Let
$$
P:\tau_1\to\tau_2
$$
be a map of colored 2-trees. According to M. Batanin
\cite{Bat} $P$ is called {\em a full injection} if $P_{S}: S_{\tau_1}\to
S_{\tau_2}$ is a color-preserving
isomorphism and $P_T : T_{\tau_1}\to
T_{\tau_2}$ is an injection. Next, let $\tau$ be a colored 2-tree with its underlying
2-tree $\t:S\to T$. We say that $\tau$ is {\em pruned} if  the map $\t:S\to T$ is surjective.

Let $\cO$ be a  unital colored 2-operad.
Consider the composition map associated with
$P$:

$$
\cO(\tau_2)\otimes \bigotimes\limits_{s_2\in
S_2} \cO(P^{-1}s_2)\to \cO(\tau_1)\,.
$$

It is clear that each $P^{-1}s_2$ is a 2-tree
of the form $\u_c$, $c\in \chrom$. Hence we
have unital maps $\bu \to \cO(P^{-1}s_2)$.
Pre-composition with  these maps gives rise
to a map
\begin{equation}
\label{themap}
\cO(\tau_2)\to \cO(\tau_1)\,.
\end{equation}

\begin{Definition}
We call $\cO$ a pruned
2-operad if for every full injection $P$
the map (\ref{themap}) is an isomorphism.
\end{Definition}

For every colored 2-tree $\tau$ there exists a
unique (up-to an isomorphism) pruned colored 2-tree
$\tau'$ together with a full injection $\tau' \to \tau$\,.
 Thus, a  pruned 2-operad  is completely
determined by prescribing its spaces for
each pruned 2-tree.

\subsubsection{Reduced 2-operads} In this section all 2-operads are non-colored.
 Let us first define the trivial 2-operad $\triv$
 by setting
$$
\triv(\tau)=\bu\,
$$
for all 2-trees  $\tau$.
Here $\bu$ is the unit of the SMC and
the operadic  composition maps
are the canonical  maps sending
tensor products of $\bu$ to $\bu$.

We say that a pruned   2-operad $\cO$ is
{\em reduced} if

--- all the unit maps $\bu\to
\cO(\u)$, $\bu\to \cO(\u)$ are
isomorphisms;

--- $\cO(\tau)=\bu$ whenever $|S_\tau|\leq 1$ so that we
have an identification
$\cO(\tau)=\triv(\tau)$ for all such $\tau$;

--- for every map $P:\tau_1\to \tau_2$ where
$|S_{\tau_1}|,|S_{\tau_2}|\leq 1$ the
corresponding operadic composition law
coincides with that of $\triv$.

Note that, equivalently, one can only
require that the conditions are the case for
pruned 2-trees $\tau$.
\subsubsection{Desymmetrization} Given a
colored  symmetric operad $\cO$, one can
define a colored 2-operad $\des\cO$ by
setting
$$
\des\cO(\tau)=\cO(S_\tau),
$$
where the coloring on the right hand side is determined
by that of $\tau$, and the operadic
composition  maps are inherited from those
of $\cO$.

\subsection{Symmetrization}\label{symmetr1}
If the SMC $\C$ has small colimits then
the functor $\des$ has a left adjoint $\sym$\,.
For many categories of higher operads the functor
$\sym$ can
be elegantly expressed using colimits \cite{Bat}.
Here we recall from \cite{Bat} a description
of the functor $\sym$ for the
category of reduced  2-operads.

For every  set $S$ we define a category $\J(S)$\,.

The objects of $\J(S)$ are pruned
$2$-trees of the form
$$
\t : S \to T\,.
$$
Morphisms are the maps between $2$-trees
which induce the identity map on $S$\,.

Notice that, although elements of a set $S$ are not ordered, choosing an object
of the category $\J(S)$ we equip $S$ with
a total order.

\noindent\\
{\bf Remark.} It is not hard to show that
for every  set $S$ the category $\J(S)$ is a poset whose opposite is called
 the Milgram poset \cite{Milgram}.

Let $\cO$ be a reduced  2-operad.

For every  set $S$ the  2-operad $\cO$
gives us an obvious (contravariant) functor from
the category $\J(S)$
to the underlying SMC $\C$\,. We denote this
functor by $\cO_S$\,.

According to Theorem 4.3 from \cite{Bat} we have
\begin{equation}
\label{sym}
\sym \cO (S) = \colim_{\J(S)} \cO_S \,.
\end{equation}

The operadic multiplications of $\sym \cO$ can
be easily obtained from those of $\cO$ using
the properties of colimits.

\subsubsection{Model structure}\label{secmodel}  Let us consider
 the category of reduced 2-operads in the category 
of complexes over the ground field $\gf$ (i.e. a dg 2-operad). 
According to Theorem 5.3 from \cite{Bat}
this category has a closed model structure uniquely determined
by the conditions that
the class of fibrations (resp.  weak equivalences) should consist of all maps $f:\cO_1\to \cO_2$ satisfying:
given any 2-tree $\t$,  the induced map of complexes
 $f:\cO_1(\t)\to \cO_2(\t)$ is  component-wise surjective (resp. is quasi-isomorphism).
Same Theorem 5.3 from \cite{Bat} implies a model structure in the category
of topological reduced 2-operads. 

Let now $\J(S)\bighat{}$ be the   category of  contravariant functors  from $\J(S)$ to the category
of complexes over $\gf$.  One has a model structure on  $\J(S)\bighat{}$ 
which is defined in a similar
way: Let $F_1,F_2\in \J(S)\bighat{}$.    The class of fibrations (resp.  weak equivalences)
 by definition  consists of all maps $f:F_1\to F_2$ satisfying:
given any 2-tree $\t\in \J(S)$,  the induced map of complexes
 $f:F_1(\t)\to F_2(\t)$ is  component-wise surjective (resp. is quasi-isomorphism).

A functor $F\in \J(S)\bighat{}$ is called {\em cofibrant} if the natural
map from the initial object $0\to F$ is cofibrant.

\begin{Lemma}\label{model} 
For every dg reduced 2-operad $\cO$ there exists a cofibrant operad
$\R\cO$ and a weak equivalence $f:\R\cO\to \cO$ such that  for every finite set, the functor
$\R\cO_S\in \J(S)\bighat{}$ is cofibrant.
\end{Lemma}

\begin{proof} 
See the proof of Theorem 7.1 in \cite{Bat1}.
\end{proof}

\subsubsection{ Algebras over colored 2-operads}
Given an $\chrom$-colored 2-operad $\cO$ and an $\chrom$-family
of objects $X_c\in \C$, $c\in \chrom$, we
define an $\cO$-algebra structure on
$\{X_c\}_{c\in \chrom}$ as a map
$$
f:\cO\to \des\ \full(\{X_c\}_{c\in\chrom}),
$$
where $\full(X)$ is the full colored
symmetric endomorphism operad of $X$. If $\cO$ is unital,
then we additionally require that $f$ matches
the units.
\subsection{Batanin's theorem}

We observe that the operad $\triv$ is
reduced and set $\R\triv \to \triv$ to be its
cofibrant resolution in the category of
reduced  2-operads.

\begin{Theorem}  [Theorem 7.2, 7.3, \cite{Bat}]\label{bat33}
The symmetric operad $\sym\R\triv$ is weakly equivalent to
the operad of little discs  if $\C$ is the 
category of topological spaces,
and to the
singular chain operad of little discs
 if $\C$ is the category of
chain complexes of $\k$-vector spaces.
\end{Theorem}

Let us sketch its proof for $\C$ being the
category of topological spaces.

First, we observe, that the  2-operad $\R\triv$
can be replaced with any weakly equivalent
one. Batanin uses 
 the Getzler-Jones 2-operad $\GJ$\,.

This  2-operad is constructed in \cite{Bat} as
a sub  2-operad of the desymmetrization
$\des(\FM)$ of the Fulton-MacPherson version $\FM$ of
little discs operad

Then, since the desymmetrization functor $\des$
admits the left adjoint $\sym$, the inclusion
$$
\GJ \hookrightarrow \des(\FM)
$$
produces the following map
$$
\sym (\GJ) \to \FM
$$
which can be shown to be an isomorphism,
hence a weak equivalence. This completes the
proof. The case when $\C$ is the category of
chain complexes is treated by applying the
singular chain functor.

\section{Swiss Cheese (SC) Operads} \label{screview}
In this section we discuss SC-modifications of the notions of colored operad and colored 2-operad.
We conclude with formulating the SC version of Batanin's theorem on the symmetrization of
the trivial 2-operad.

\subsection{Symmetric Swiss Cheese type operads}
In this subsection we recall from \cite{Bat} the notion of
the symmetric Swiss Cheese type operads. Here, we call them
symmetric SC operads for short.

Let $\chrom:=\{\a,\c\}$ be the set of colors. {\em  An SC-set} is a
collection of the following data:

---a finite set $S$; 

---a map $\chi_S:S\to \chrom$;

---an element  $c_S\in \chrom$.

These data should satisfy:

--- if $\a\in \chi(S)$, then $c_S=\a$.

{\em A map of SC-sets}  is a usual map $P:S_1\to S_2$ satisfying:
if $s_1\in S_1$ is such that $\chi_{S_1}(s_1)=\a$, then
$\chi_{S_2}(P(s_1))=\a$.
Given such  $P$ and $s_2\in S_2$, $P^{-1}s_2$ is naturally an
SC-set: $\chi_{P^{-1}s_2}$ is the restriction
of $\chi_{S_1}$; $c_{P^{-1}s_2}=\chi_{S_2}(s_2)$.

{An SC-operad in a symmetric monoidal category $\C$}
 is a functor $\cO$ from the groupoid of SC-sets and
 their color preserving bijections to $\C$.

For every map $P:S_1\to S_2$ of SC-sets, there should be given a composition map
$$
\cO(S_2)\otimes\bigotimes\limits_{s_2\in S_2}\cO(P^{-1}s_2)\to\cO(S_1)\,.
$$
These compositions should satisfy the associativity law which is  similar to that for
usual operads. It is clear how to define {\it unital} symmetric SC operads.
In this paper all our symmetric SC operads are unital.

\subsubsection{Reduced symmetric SC operads}
\label{reduced-SC-symm}
We say that a unital symmetric SC operad $\cO$ is {\it reduced}
if for every SC set $S$ with at most one element
$$
\cO(S) \cong \bu\,.
$$
For the one element SC sets these isomorphisms
should coincide with the unit maps.
Furthermore, the operadic compositions of
zero-ary and unary operations send products
of $\bu$ to $\bu$ via the corresponding isomorphism
of the symmetric monoidal category.

\subsubsection{Colored symmetric SC-operads}
Fix two sets of colors $\chrom_c$
and $\chrom_a$. {\em A colored SC-set $S$} is a map
$\chi_S:S\to \chrom_c\sqcup \chrom_a$ and an element
$c_S\in \chrom_c\sqcup \chrom_a$ satisfying: if $\chi_S^{-1}\chrom_a$ is non-empty,
then $c_S\in \chrom_a$.

We declare that there are no maps between colored SC-sets $S_1$ and
$S_2$ if $c_{S_1} \neq c_{S_2}$\,. On the other hand if
$c_{S_1} = c_{S_2}$ then a map from $S_1$ to $S_2$ is a map of
sets $P:S_1\to S_2$ satisfying the property:
{\it for any $s_2\in S_2$,
the set $P^{-1}s_2$ along with
the map $\chi_{S_1 }|_{P^{-1}s_2}:P^{-1}s_2\to \chrom_c\sqcup \chrom_a$ and
the element $c_{P^{-1}s_2 }:=\chi_{S_2}(s_2)$ is a colored SC-set.}
Thus, given a map of colored SC-sets $P:S_1\to S_2$ and $s_2\in S_2$, we
have a colored SC-set $P^{-1}s_2$.

{\em A colored SC-operad $\cO$ in a SMC $\C$} is a functor $\cO$ from
the isomorphism groupoid of colored SC-sets to
$\C$ along with the composition maps: given  a map $P:S_1\to S_2$ of colored sets,
one should have a map
$$
\cO(S_2)\otimes \bigotimes\limits_{s_2\in S_2}\cO(P^{-1}s_2)\to \cO(S_1)
$$
satisfying the associativity property as above.

The operad $\sO$ is an example of colored SC operad.
Indeed, let $\chrom_c:=\nat$ and $\chrom_a:=\{\a\}$
and let $S$ be a colored SC-set. Let $S_{\c}:=\chi^{-1}\chrom_c$ and
$S_{\a}:=\chi^{-1}\chrom_a$.

In the case $c_S\in \chrom_c$, set
$$
\sO(S):=\sO(S)^{c_S}_{\{\chi(s)\}_{s\in S}};
$$
if $c_S=\a$, we set
$$
\sO(S):=\sO(S_{\c},S_{\a})_{\{\chi(s)\}_{s\in S_{\c}}}.
$$

\subsection{SC 2-operads}
Let us review Batanin's definition of a Swiss Cheese type (or simply SC) 2-operad
from \cite{Bat}.
This notion is obtained via modifying the definition of a usual 2-operad as follows:

1) {\em An SC-ordinal} is any non-empty ordinal; its minimum is considered to be marked.

2) {\em A map of SC-ordinals} is a monotonous map preserving the minima.

3) {\em An SC 2-tree $\t$} is a monotonous map $\t:S\to T$ where $S$ is a usual ordinal
and $T$ is an SC-ordinal.
A map of SC 2-trees
$$
(\t:S\to T)\to (\t_1:S_1\to T_1)
$$
is a map of sets $P_S:S\to S_1$ as well as a map of SC-ordinals
$P_T:T\to T_1$ such that $\t_1P_S=P_T\t_2$; the map $P_S$ must preserve
the order on each set
$\t^{-1}t$, $t\in T$.
Given $s_1\in S_1$ such that $\t_1(s_1)$ is not the minimum of $T_1$, we define
a usual 2-tree $P^{-1}s_1$ in the same way as for the usual 2-trees
(see the beginning of Sec. \ref{review}); in the case $\t_1(s_1)$ is the
minimum of $T_1$, we naturally get an SC 2-tree $P_S^{-1}s_1$.

4) We define {\em an SC 2-operad in a symmetric monoidal
 category $\C$} as:

--- a functor
$$
\cO: \mbox{\bf{2-trees}}^\times \sqcup \mbox{\bf{SC 2-trees}}^\times\to \C;
$$

--- for every map of 2-trees or SC 2-trees
$P:\t\to \t_1$, there should be given a map
$$
\cO(\t_1)\otimes \bigotimes\limits_{s_1\in S_1} \cO(P^{-1}s_1)\to \cO(\t).
$$

These  maps should satisfy the associativity property which is similar to that
for usual  2-operads.

\subsubsection{Unital SC 2-operads}
In this paper all SC 2-operads are assumed to be unital.

To introduce the notion of unital SC 2-operads we define
$\u_{\c}$ to be the ordinary 2-tree $\pt \to \pt$\,.
We also define $\u_{\a}$ to
be an SC 2-tree in which a 1-element ordinal is mapped into a
one-element SC ordinal.

For every isomorphism $P:\t_1\to \t_2$ of 2-trees or
SC 2-trees every pre-image
$P^{-1}s_2$, $s_2\in S_{\t_2}$, is
either $\u_{\c}$ or $\u_{\a}$.
For every 2-tree $\t$ the pre-image $Q_{\c}^{-1}\, \pt$ of the point for
a unique map $Q_{\c} : \t \to \u_{\c}$ is equal to $\t$\,.
Furthermore, for every SC 2-tree $\t$ the pre-image $Q_{\a}^{-1}\, \pt$ of
the point for a unique map $Q_{\a} : \t \to \u_{\a}$ is also equal
to $\t$\,.

Define {\em a unital SC 2-operad } as an SC
2-operad $\cO$ along with maps $\bu\to
\cO(\u_{\c})$ and $\bu \to \cO(\u_{\a})$\,.
satisfying:

--- for every isomorphism $P:\t_1 \to \t_2$,
of 2-trees or SC 2-trees the map
$$
\cO(\t_2)\cong \cO(\t_2)\otimes {\bu
}^{\otimes S_{\t_2}}\to \cO(\t_2)\otimes
\bigotimes\limits_{s_2\in S_{\t_2}}
\cO(P^{-1}s_2)\to \cO(\t_1)
$$
coincides with the map $\cO(\t_2)\to
\cO(\t_1)$ induced by $P^{-1}$ from the
definition of  $\cO$ as a functor from
the corresponding groupoid.

--- for every 2-tree $\t$  the composition
$$
 \cO(\t) \cong \bu \otimes \cO(\t) \to
\cO (\u_{\c}) \otimes \cO(\t) \to \cO(\t)
$$
is the identity on $\cO(\t)$\,.

--- for every SC 2-tree $\t$  the composition
$$
 \cO(\t) \cong \bu \otimes \cO(\t) \to
\cO (\u_{\a}) \otimes \cO(\t) \to \cO(\t)
$$
is the identity on $\cO(\t)$\,.

\subsubsection{} We define the trivial SC
2-operad $\triv$ by setting
$$
\triv(\t)=\bu\,,
$$
for all 2-trees and SC 2-trees $\t$\,.
Here $\bu$ is the unit of the SMC and
the operadic multiplications
are the canonical  maps sending
tensor products of $\bu$ to $\bu$.

\subsubsection{Colored SC 2-operads}\label{chic}
 We define  a colored  SC 2-operad as follows. Fix 2 sets of colors:
$\chrom_c$ and $\chrom_a$. Define a coloring of
 an SC 2-tree $\t_\tau:S_\tau\to T_\tau$ as follows.

First decompose $S_\tau = S_{\tau,a} \sqcup S_{\tau,c}$, where $S_{\tau,a}$
is the $\t_\tau$-preimage of the minimum of $T_\tau$, and $S_{\tau,c}$
is the complement.

{\em A $(\chrom_c,\chrom_a)$-coloring of $\tau$} is
 a
prescription of  maps $\chi_{\tau,c}:S_{\tau,c}\to \chrom_c$;  $\chi_{\tau,a}:
S_{\tau,a}\to \chrom_a$ and an element $c_\tau\in \chrom_a$.

As well as for ordinary colored 2-trees we declare that there are
no maps between colored SC 2-trees $\tau$ and $\tau_1$ if
$c_{\tau} \neq c_{\tau_1}$\,. On the other hand, if $c_{\tau} = c_{\tau_1}$
then a map $P:\tau\to\tau_1$ is just the map of the underlying
SC 2-trees. Then it is clear that
for every $s_1\in S_{\tau_1}$ such that $\t_{\tau_1}s_1$ is the minimum, the
SC 2-tree $P^{-1}s_1$ is naturally $(\chrom_c,\chrom_a)$-colored.
Furthermore, for every $s_1\in S_{\tau_1}$ such that $\t_{\tau_1}s_1$ is not
the minimum, the 2-tree $P^{-1}s_1$ is naturally $\chrom_c$-colored.

We define a $(\chrom_c,\chrom_a)$-colored SC 2-operad as a functor $\cO$ from
the disjoint union of the groupoid of $\chrom_c$-colored 2-trees and the groupoid of
$(\chrom_c,\chrom_a)$-colored SC 2-trees to $\C$. Given a map
$P:\tau\to \tau_1$ of $\chrom_c$-colored 2-trees or $(\chrom_c,\chrom_a)$-colored SC
2-trees there should be given a map
$$
\cO(\tau _1)\otimes \bigotimes\limits_{s_1\in S_1} \cO(P^{-1}s_1)\to \cO(\tau).
$$

The associativity axiom should be satisfied.

\subsubsection{Unital colored SC 2-operads}
As well as SC 2-operads all colored SC 2-operads
are assumed to be unital.

To introduce the notion of unital colored SC 2-operads we define
$\u_c$, $c\in \chrom_c$, be the colored 2-tree $\pt \to \pt$
for which the point $\pt$ has the color $c$ and $c_{\u_c} = c$\,.
Similarly, we define $\u_a$, $a\in \chrom_a$ to be the colored SC 2-tree
in which the one-element ordinal
is mapped into the one-element SC ordinal and all colorings are $a$.

For every isomorphism $P:\tau_1\to \tau_2$ of $\chrom_c$-colored
2-trees or SC 2-trees every pre-image
$P^{-1}s_2$, $s_2\in S_{\tau_2}$, is
either $\u_{c}$ or $\u_{a}$.
For every colored 2-tree or colored SC 2-tree $\tau$
the pre-image $Q_{\tau}^{-1}\, \pt$ of the point for
a unique map $Q_{\tau} : \tau \to \u_{c_{\tau}}$ is equal to $\tau$\,.

Define {\em a unital colored SC 2-operad } as a colored SC
2-operad $\cO$ along with maps $\bu\to
\cO(\u_{c})$ and $\bu \to \cO(\u_{a})$ for all
$c\in \chrom_c$ and $a\in \chrom_a$
satisfying:

--- for every isomorphism $P:\tau_1 \to \tau_2$,
of colored 2-trees or colored SC 2-trees the map
$$
\cO(\tau_2)\cong \cO(\tau_2)\otimes {\bu
}^{\otimes S_{\tau_2}}\to \cO(\tau_2)\otimes
\bigotimes\limits_{s_2\in S_{\tau_2}}
\cO(P^{-1}s_2)\to \cO(\tau_1)
$$
coincides with the map $\cO(\tau_2)\to
\cO(\tau_1)$ induced by $P^{-1}$ from the
definition of $\cO$ as a functor from
the corresponding groupoid.

--- for every colored 2-tree or colored SC 2-tree $\tau$
the composition
$$
 \cO(\tau) \cong \bu \otimes \cO(\tau) \to
\cO (\u_{c_{\tau}}) \otimes \cO(\tau) \to \cO(\tau)
$$
is the identity on $\cO(\tau)$\,.

\subsubsection{Pruned SC 2-operads} A (colored) SC 2-tree $\tau$ is called pruned
if $\text{Im}(\t_\tau)\supset
T_\tau\backslash m_{T_\tau}$, where $m_{T_\tau}$ is the marked minimum of $T_\tau$.

For every colored SC 2-tree $\tau$ there exists a unique up to isomorphism pruned
colored SC 2-tree $\tau'$ and a map $P:\tau'\to \tau$ such that
$P_S:S_{\tau'}\to S_{\tau}$ is a bijection; $P_T$ is injective, and $P$ induces an
isomorphism of colorings. For every such $P$ the pre-images $P^{-1}s$
are of the form $\u_c$ or $\u_a$, therefore, given a unital operad $\cO$,
we have a map
\begin{equation}
\label{map}
\cO(\tau')\to \cO(\tau)\,.
\end{equation}
By analogy with ordinary 2-operads (see Subsection \ref{pruned})
$\cO$ is called {\em pruned} if all such
maps (\ref{map}) are isomorphisms.

\subsubsection{Reduced SC 2-operads}
We say that a pruned  (non-colored) SC 2-operad $\cO$ is
{\em reduced} if

--- all the unit maps $\bu\to
\cO(\u_c)$, $\bu\to \cO(\u_a)$ are
isomorphisms;

--- $\cO(\tau)=\bu$ whenever $|S_\tau|\leq 1$ so that we
have an identification
$\cO(\tau)=\triv(\tau)$ for all such $\tau$;

--- for every map $P:\tau_1\to \tau_2$ where
$|S_{\tau_1}|,|S_{\tau_2}|\leq 1$ the
corresponding operadic composition law
coincides with that of $\triv$.

Note that, equivalently, one can only
require that the conditions are the case for
pruned 2-trees $\tau$.

\subsection{Desymmetrization}\label{coloredSC}
Given a symmetric SC-operad
$Q$, Batanin defines its desymmetrization $\des Q$ by setting
$\des Q(\t):= Q(S_\t)$ for all 2-trees and SC 2-trees $\t$\,.
Here $S_\t$ is treated as an SC-set as follows:

--- if $\t$ is a usual 2-tree then we define  all the colorings to be $\c$;

--- if $\t$ is an SC 2-tree, we set $\c_{S_\t}:=\a$ and we give the preimage
of marked element of $T_\t$ the color $\a$, the remaining elements of $S_\t$
receive color $\c$.

Given a reduced symmetric
SC-operad $Q$, its desymmetrization
$\des Q$ is a  reduced SC 2-operad  so that
 $\des$ is a functor from the category of reduced symmetric
SC operads to that of reduced SC 2-operads.

\subsubsection{}  In the same spirit, one defines the desymmetrization of a
colored symmetric SC-operad $\cO$.
Let $\tau$ be a colored SC 2-tree; we then see that $S_\tau$ is
a colored SC-set in the natural way:
the map
$\chi_{S_\tau}\Big|_{S_{\tau,c}}:=\chi_{\tau, c}$
and $\chi_{S_\tau}\Big|_{S_{\tau,a}}:=\chi_{\tau, a}$. Finally,
$c_{S_\tau}:=c_\tau$.
We then set
$$
\des(\cO)(\tau):=\cO(S_\tau)
$$
with the composition law determined by that in $\cO$.

\subsection{Symmetrization} The content of this section is a straightforward SC generalization
of Sec \ref{symmetr1}.

Under an assumption that  SMC $\C$ has small colimits,
the functor $\des$ has a left adjoint $\sym$\,.
We have  a description
of the functor $\sym$ for the
category of reduced SC 2-operads which is similar to that for
reduced 2-operads (see. Sec \ref{symmetr1}).

For every SC set $S$ we define a category $\J(S)$\,.

If $c_{S} = \c$, 
 then the category $\J(S)$ is the same as in Sec \ref{symmetr1}:
the objects of $\J(S)$ are pruned
$2$-trees of the form
$$
\t : S \to T\,.
$$
Morphisms are the maps between $2$-trees
which induce the identity map on $S$\,

If $c_{S} = \a$ then objects of $\J(S)$ are
pruned SC 2-trees $\t : S\to T$ such that the preimage of
the minimal element of $T$ coincides with $S_{\a}= \chi^{-1}(\a)$\,.
Morphisms are the maps between SC $2$-trees
which induce the identity map on $S$\,.

As in the non SC case,  all categories $\J(S)$ are in fact posets.

Let $\cO$ be a reduced SC 2-operad.
For every SC set $S$ the operad $\cO$
gives us an obvious (contravariant) functor from
the category $\J(S)$
to the underlying SMC $\C$\,. We denote this
functor by $\cO_S$\,.

According to Theorem 9.1 from \cite{Bat} we have
\begin{equation}
\label{sym-SC}
\sym \cO (S) = \colim_{\J(S)} \cO_S \,.
\end{equation}

The operadic multiplications of $\sym \cO$ can
be easily obtained from those of $\cO$ using
the properties of colimits.

\subsubsection{Model structure} The category of dg pruned SC 2-operads
has a model structure which is defined in the same way as in Sec \ref{secmodel} 
Same is true for the model structure on the category $\J(S)\bighat{}$ of contravariant
functors from $\J(S)$ to the category of compexes over $\gf$.

Lemma \ref{model} holds true in the SC context.
\begin{Lemma}\label{scmodel} 
For every dg reduced SC 2-operad $\cO$ there exists
a cofibrant dg reduced SC 2-operad $\R\cO$ and a quasi-isomorphism $f:\R\cO\to \cO$
such that the functors $\R\cO_S\in \J(S)\bighat{}$ are cofibrant
for any SC set $S$.
\end{Lemma}
\begin{proof} Similar to the proof of Lemma \ref{model}.
\end{proof}

\subsubsection{Batanin's theorem}

\begin{Theorem} [Theorem 9.2, 9.4, \cite{Bat}]\label{bat3}
The symmetric SC operad $\sym\R\triv$ is weakly equivalent to
Voronov's Swiss Cheese operad if $\C$ is the
category of topological spaces,
and to the
singular chain operad of Voronov's Swiss
Cheese operad if $\C$ is the category of
chain complexes of $\k$-vector spaces.
\end{Theorem}

Batanin's proof goes along the same lines as his proof of Theorem \ref{bat33}.

\section{Linking the operad $\se$ with 2 operads:
a 2-operad $\seq$}\label{seq}
In this section we will define a 2-sub-operad $\seq\subset\des\se$. Next, we  define the SC-version
of $\seq$. 
\subsection{2-operad $\seq$ (cochain part)}
Let us first recall the definition of the cochain part  (i.e. 'non-SC part') of the $\nat$-colored operad $\se$ (Sec \ref{modif})
with the notation slightly changed.

Let 
 $S$ be a finite set and $J$; $I_s,s\in S$ be non-empty finite ordinals. Each element $u$
of  the operadic space
$\se(S)^J_{\{I_s\}_{s\in S}}$
is defined by means of the following data:

--- a total order $>_u$ on the set
 $$\I:=\bigsqcup\limits_{s\in
S} I_s;
$$
a non-decreasing map
$$
Q_u:\I\to J.
$$

 These data should satisfy:

i) the order $>_u$  on $\I$ agrees with those on
each $I_s$;

ii) same as condition 3) from Sec \ref{posle}.

The composition law for the operad $\se$  was defined in Sec \ref{compse}.

Let us now define a colored sub -2-operad $\seq$ of $\des \se$.  Let $\tau$ be a $\nat$-colored
2-tree, which is defined by means of a 2-tree $\t:S\to T$ and its $\nat$-coloring
such that an $s\in S$ has color
$I_s$, where $I_s$ is a
non-empty finite ordinal, and the color of
the result is $J$. More formally, $\chi_\tau(s)=I_s$; $c_\tau=J$.

Let us define  subsets
$$
\seq(\tau):=\seq(\t)^{J}_{\{I_s\},s\in S}\subset  \se(S)^{J}_{\{I_s\},s\in S}
$$
 which consists of all elements $u\in  \se(S)^{J}_{\{I_s\},s\in S}$ satisfying:

--- if $i,k\in I_{s_1}$, $j\in I_{s_2}$,
$s_1\neq s_2$ and $i<_u j<_u k$, then
$\t(s_2)<\t(s_1)$;

--- if $s_1,s_2\in S$, $s_1<s_2$, and
$\t(s_1)=\t(s_2)$, then $I_{s_1}<_u
I_{s_2}$.

One can check  that thus defined subspaces are closed under  all 2-operadic composition maps
so that $\seq\subset \des\se$ is a colored sub-2-operad.   Thus defined colored sub-2-operad coincides with the
colored 2-operad $\seq$ as defined in Sec 6.1 of \cite{dgcat}.  The check that 
$\seq$ is closed  under the 2-operadic compositions follows from the observation that the 2-operadic
composition maps inherited from $\se$ are the same as in loc. cit.

\subsection{SC version of $\seq$}\label{scseq2}
Let us define a colored SC 2-operad $\scseq$ by modifying
 the definition of $\seq$ as follows.

First of all we fix the sets of colors:

--- the set $\chrom_c$ is the same as the set
of colors of $\seq$, i.e. $\nat$;

--- the set $\chrom_a$ is the one element set $\{\a\}$; we
identify a unique element of $\chrom_a$ with the ordinal
consisting of 1 element.

--- Given a usual colored 2-tree $\tau$ we set
$\scseq(\tau):=\seq(\tau)$;

--- given a colored SC 2-tree $\tau$, let us construct
a usual colored 2-tree $\tau'$ with the underlying
2-tree $\t' = \t: S \to T$\,. Define
a map $\chi_{\tau'}: S \to \chrom_c = \nat$ by setting

a) if $s\in S$ and $\t(s)$ is the minimum of $T$,
then we set $\chi_{\tau'}(s)$ to be the
 one-element ordinal;

b) if $s\in S$ and $\t(s)$ is not the minimum of
 $T$, then we set $\chi_{\tau'}(s)=\chi_{\tau,c}(s)$\,,
where $\chi_{\tau,c}$ is a defining map of the
coloring for $\tau$ (see Sec. \ref{chic}).

Lastly, we set $c_{\tau'}$ to be the one-element ordinal.

We then define $\scseq(\tau):=\seq(\tau')$.

Note that we have natural inclusions
$$
\scseq(\tau) \subset \se(S_{\tau})=(\des \se)(\tau),
$$
where $S_{\tau}$ is the colored SC-set corresponding to the colored
2-tree or the colored SC 2-tree $\tau$
as defined in Sec \ref{coloredSC}.
Thus $\scseq$ is a colored SC 2-suboperad of
$\des \se$.

\
\subsection{Totalization: A  dg 2-operad $|\seq|$ and a dg operad $\sco$}\label{condens}
Using the functor of (co)-simplicial totalization we will convert a colored  operad
$\se$ and  a colored  2-operad $\seq$ 
into differential graded operads.  The SC versions will be covered in the next section \ref{condenssc}.

The spaces of unary 
operations in  $\seq$ and $\se$ give a
category structure on $\nat$
$$
\hom(I_1,I_2)=\seq(t_0)_{I_1}^{I_2}=\se_{I_1}^{I_2}\,,
$$
where $t_0:\pt\to\pt$. This category is
isomorphic to the simplicial category
$\Delta$.

The action of these unary operations defines
a polysimplicial/cosimplicial structure on the collection
of operadic sets.

Given a 2-tree $\t:S\to T$,
the collection of sets 
$$
\seq(\t)^J_{\{I_s\}_{s\in
S}},
$$
where $I_s,J$ are non-empty final ordinals,
forms a functor 
$$\seq(\t):\Delta\times (\Delta^\opp)^S\to \Sets,
$$
where $J\in \Delta$ and $I_s\in \Delta^\opp$.

Using the functor
$\S:\Delta\to\complexes$ we can take the
total complexes of these polysimplicial (cosimplicial)
sets, in the same way as in Subsection \ref{202}.
We set
$$
|\seq|(\t):=|\seq(\t)|;
$$
The complexes $|\seq|(\t)$ automatically form a dg-operad.  

Similarly, the sets $\se(S)^J_{\{I_s\}_{s\in S}}$ form a functor
$$\se(S):\Delta\times (\Delta^\opp)^S\to\Sets$$
 so that we can define
$$
|\se|(S):=|\se(S)|.
$$

We have a dg-operad structure on $|\se|$. The embedding $\seq\subset \des\se$ induces a map
$$
|\seq|\to \des|\se|.
$$
\subsection{Extension to the SC-setting: a dg SC 2-operad $\scseq$
and an SC operad $|\se|$}\label{condenssc}
Let us  now extend the construction of $|\seq|$ and $|\se|$ to the SC case.

Given an SC 2-tree $\t:S\to S_1$, let us decompose
$S=S_\c \sqcup S_\a$, where $S_\a$ is the pre-image of the minimum of $S_1$.
Suppose we are given
ordinals $I_s,s\in S_\c$. Using these data, we
 naturally get a colored SC 2-tree
$\tau:=\tau(\t,\{I_s\}_{s\in S_\c})$, where the coloring sets are
$\chrom_c=\nat$ and $\chrom_a=\{\a\}$. Each element of $s\in S_{\c}$ receives
color $I_s$; each element of
$S_{\a}$ gets colored in $\a$; we set $c_\tau=\a$.

We set
$$
\scseq(\t)_{\{I_s\}_{s\in
S_{\c}}}:=\scseq(\tau).
$$

Thus,
given an SC 2-tree $\t:S\to S_1$, we get a polysimplicial set
$$
\scseq(\t): (\Delta^\opp)^{S_{\c}}\to \Sets\,.
$$

Set $|\scseq|(\t):=|\scseq(\t)|$. For $\t$ being a 2-tree we set $|\scseq|(\t):=|\seq|(\t)$.
 This way the dg  2-operad $|\seq|$ extends to an SC 2-operad
$|\scseq|$.

Given an SC set $S=S_\c\sqcup S_\a$ and ordinals $I_s$, $s\in S_{\c}$ we get 
a $\nat$-colored SC-set
which determines the operadic set
$$
\se(S_\c,S_\a)_{\{I_s\}_{s\in S_\c}}.
$$
These sets form a functor
$$
\se(S): (\Delta^\opp)^{S_\c}\to
\Sets.
$$
and we can set $|\se|(S):=|\se(S)|$ thereby getting an SC symmetric operad $|\se|$ which is an  SC exstension
of the symmetric operad $|\se|$ from the previous section \ref{condens}.

The map $\scseq\to \des\se$ of SC 2-operads induces a map
\begin{equation}\label{sc-and-sco}
|\scseq|\to 
\des|\se|
\end{equation}
of dg SC 2-operads.

Since the operad $\se$ is isomorphic to $\op$
the DG operad $\sco$ is isomorphic to the DG operad
$\La = |\op|$ of natural operations on the pair
$(C^{\bul}(A,A) ; A)$\,. Thus we have a map from 
the SC 2-operad $|\scseq|$ to $\des \La$\,.

\section{The SC 2-operad $\scR$ and the SC operad $\scoR$}\label{braces}

It is not hard to see that the
SC 2-operad $\sc$ is pruned.
However, neither $\sc$ nor $|\se|$
is reduced.
In this section we construct a reduced SC 2-operad
$\scR$ which is quasi-isomorphic to the SC 2-operad
$\sc$\,. Similarly, we construct a reduced SC operad
$\scoR$  which is quasi-isomorphic to the SC operad
$|\se|$\,. Both $\scR$ and $\scoR$ are obtained as
suboperads of $\sc$ and $|\se|$, respectively.

As usual we will first make all definition for the non-SC  part and then extend them to
the SC-setting

\subsection{An increasing filtration on the colored 2-operad $\seq$}
\label{ordeq}
Let $\t$ be a 2-tree $\t:S\to T$ and
$$
v\in \seq(\t)_{\{I_s\}_{s\in S}}^{J}\,.
$$

Consider the order $>_v$ on
\begin{equation}
\label{I-S0}
\I := \I_{S}:=\bigsqcup\limits_{s\in S} I_s\,.
\end{equation}
Call two elements $i_1, i_2 \in \I_S$ {\em elementary equivalent}
if $i_1, i_2 \in I_s$ for some $s\in S$ and
for every $i \in \I_{S}$ between $i_1$ and $i_2$ with respect
to the order $<_v$ the element $i$ belongs to $I_s$\,.
In this way we get an equivalence relation on $\I_S$. Denote
by $|v|$ the number of equivalence classes
with respect to this relation.

Let $F_N\seq(\t)_{\{I_s\}_{s\in S}}^J$
 be the subset consisting of all elements
$v$ with $|v|\leq N+|S|$. Roughly speaking,
the difference $|v|- |S|$ counts how many times
the order $<_v$ cuts the ordinals $I_s, ~ s\in S$
into subordinals.

\subsection{Extension of the filtration onto $\scseq$}
\label{ordeq-SC}
Let $\t:S\to T$ be an SC 2-tree.
As above, we set
$S_{\a}$ to be the pre-image of the minimum of $T$ and
$S_{\c}:=S\setminus S_{\a}$.

Recall that an element
$$
v\in \scseq(\t)_{\{I_s\}_{s\in S_{\c}}}
$$
is nothing else but a total order $>_v$ on
\begin{equation}
\label{the-whole-thing}
\bigsqcup\limits_{s\in S_{\c}} I_s\sqcup S_{\a}
\end{equation}
subject to certain conditions.

In order to define the elementary equivalence relation
on (\ref{the-whole-thing}) we replace (\ref{the-whole-thing})
by the isomorphic set
\begin{equation}
\label{I-S}
\I_{S_{\c} \sqcup S_{\a}} = \bigsqcup\limits_{s\in S} I_s\,,
\end{equation}
where $I_s$ is the one element
ordinal for every $s\in S_{\a}$.

Using the total order $>_v$ on $\I_{S_{\c} \sqcup S_{\a}}$ and
the construction from the previous subsection we get the
elementary equivalence relation on the set
$\I_{S_{\c} \sqcup S_{\a}}$ and hence on (\ref{the-whole-thing}).

On the set (\ref{the-whole-thing}) the elementary
equivalence relation can be described as follows. The restriction
of this relation onto $S_{\a}$ coincides with
the identity relation, there is no element of $S_{\a}$ which
is equivalent to an element
$$
i \in \bigsqcup\limits_{s\in S_{\c}} I_s\,.
$$
Finally we call two elements
$$
i_1, i_2 \in \bigsqcup\limits_{s\in S_{\c}} I_s
$$
elementary equivalent iff

--- $i_1, i_2\in I_s$ for some $s\in S_{\c}$\,,

--- for every element $i$ of the set (\ref{the-whole-thing})
between $i_1$ and $i_2$
with respect to the order $<_v$ we have $i \in I_s$\,.

We denote the
number of equivalence classes in (\ref{I-S}) $\I_{S_{\c} \sqcup S_{\a}}$
 by $|v|$ and define the subset
$$
F_N\scseq(\t)_{\{I_s\}_{s\in S_{\c}}}\subset
\scseq(\t)_{\{ I_s\}_{s\in S_{\c}}}
$$
to consist of all elements $v$ with $|v|\leq
N+|S|$.
\subsection{Compatibility of the filtration with the operadic structure}

\begin{Lemma}\label{lmfil} The filtration $F$ is
compatible with operadic compositions on
$\scseq$.
\end{Lemma}
\begin{proof}
Let us first prove Lemma for the filtration on  2-operad $\seq$ (i.e. the 'non-SC'-part).
Consider operadic compositions of the
following type:

Let $P:\t_1 \to \t_2$ be a map of 2-trees,
where $\t_1:S_1\to T_1$ and $\t_2:S_2\to T_2$\,.
Let $P_S:S_1\to S_2$ be the induced map. For
every $s_2\in S_2$, we have a pre-image
$P_S^{-1} (s_2) \subset S_1$\,.

Let $I_{s_1}$, $s_1\in S_1$; $J_{s_2}$,
$s_2\in S_2$ ; $J$ be non-empty ordinals.

Let
$$
w\in \seq(\t_2)^{J}_{\{J_{s_2}\}_{s_2\in S_2}};
$$
$$
v_{s_2}\in
\seq(P^{-1}s_2)_{\{I_{s_1}\}_{s_1 \in
P_S^{-1}(s_2)}}^{J_{s_2}}\,.
$$

Let us denote by $z$ the composition of these elements
and estimate $|z|$.  Suppose that
$$
J_{\vs}\subset (\bigsqcup\limits_{s_2\in
S_2} J_{s_2},>_w)
$$
for $\vs \in S_2$ is split into $|\vs|$ equivalence classes.

Consider the map
$$
\I_{\vs}:=\bigsqcup\limits_{s_1\in
P^{-1}\vs} I_{s_1}\to J_\vs\,.
$$

It is clear that the number of equivalence
classes of
$$
\I_\vs\subset
(\bigsqcup\limits_{s_1\in S_1}
I_{s_1},>_z)
$$
does not exceed $|v_\vs|+|\vs|-1$. Therefore
$$
|z|\leq \sum\limits_{\vs\in S_2}
(|v_\vs|+|\vs|-1) = |w|-|S_2|+\sum\limits_\vs
|v_\vs|\,.
$$
Hence,
$$
|z|-|S_1|\leq |w|-|S_2|+\sum\limits_\vs
(|v_\vs|-|P^{-1}\vs|)
$$
which means that this composition is
compatible with the filtration $F$.
This concludes the proof for  $\seq$.
 The extension to $\scseq$ is straightforward.
\end{proof}

This Lemma, in particular implies that the
polysimplicial/cosimplicial structure on 
$\scseq$ is
compatible with the filtration $F$.
 Therefore, the filtration $F$ descends onto
 the level of total complexes so that
  we have an increasing filtration on each
operadic complex $\sc(\t)$:
$F_N\sc(\t)\subset \sc(\t)$.

\begin{Lemma}
\label{lmfil1}
The filtration $F$ on $\sc$ satisfies
the following properties:
\begin{enumerate}
\item The operadic compositions in $\sc$ are
compatible with the filtration.

\item The complex $F_N\sc(\t)$ is
concentrated in the degrees $\geq -N$\,.

\item The quotient
$F_N\sc(\t) \,\Big/\, F_{N-1} \sc(\t)$ only has  cohomology
concentrated in degree $-N$.
\end{enumerate}
\end{Lemma}

\begin{proof}

~\\
(1) Follows from the previous lemma.

~\\
(2) Let us first consider the non-SC part of the statement (that is, we will prove the statement
for $\seq$).  Let $\t:S\to T$ be a 2-tree.
Let us consider the simplicial realization
with respect to the lower indices for
\begin{equation}
\label{gadget}
\seq(\t)^J_{\{I_{s}\}_{s \in S}}\,.
\end{equation}
Let
$$
v \in \seq(\t)^J_{\{I_{s}\}_{s \in S}}\,.
$$

According to Subsection \ref{ordeq}
the order $>_v$ on
\begin{equation}
\label{I-S01}
\I_S:=\bigsqcup\limits_{s \in S}
 I_{s}
\end{equation}
defines on $\I_S$ an equivalence relation.

If
\begin{equation}
\label{ineq}
|v|+|J|-1 < \sum_{s\in S} |I_s|\,.
\end{equation}
then there exist
two different but equivalent elements of $\I_S$ which go to
the same element in $J$\,. In this case the element $v$
is obtained from another element by applying a degeneracy.

Thus if inequality (\ref{ineq}) holds for $v$ then
$v$ does not contribute to the realization of (\ref{gadget}).

Therefore if $v$ contributes to the realization
then
$$
|J|-1 -\sum_{s\in S} (|I_s|-1) \geq -|v|+|S|
$$
and hence the complex
$$
F_N |\seq|(\t)
$$
is concentrated in degrees
$$
\geq - N\,.
$$
This finishes the proof for $|\seq|$.

The general SC-case is similar. 
Let $\t : S \to T$ be an SC $2$-tree.
Let $S_{\a}$ be the pre-image of the minimal element of $T$
and $S_{\c} = S\, \setminus\, S_{\a}$\,. An element $v$ of
\begin{equation}
\label{gadget-SC}
\scseq(\t)_{\{I_{s}\}_{s \in S_{\c}}}
\end{equation}
is a total order $>_v$ on
\begin{equation}
\label{the-whole-thing1}
\bigsqcup_{s\in S_{\c}} I_s\, \sqcup \, S_{\a}
\end{equation}
subject to certain conditions.

According to Subsection \ref{ordeq-SC}
the order $>_v$ gives us the elementary equivalence
relation on the set (\ref{the-whole-thing1}).

If at least one equivalence class in (\ref{the-whole-thing1})
contains more than 1 element then the corresponding element
$v$ in (\ref{gadget-SC})
is obtained from another element by applying a degeneracy.
Indeed, only the equivalence classes in
$$
\bigsqcup_{s\in S_{\c}} I_s
$$
may contain more than one element.
And if at least one class contains more than 1 element
then there are distinct elements $i_1, i_2\in I_s$ for
some $s\in S_{\c}$ such that one of them goes right after
another in the ordinal (\ref{the-whole-thing1}).

Therefore, if $v$ contributes to the realization of (\ref{gadget-SC})
then
$$
\sum_{s\in S_{\c}} |I_s| + |S_{\a}| = |v|\,.
$$
Hence
$$
\sum_{s\in S_{\c}} (|I_s|-1) = |v|- |S_{\a}| - |S_{\c}|
$$
or equivalently
$$
- \sum_{s\in S_{\c}} (|I_s|-1) = |S| - |v|\,.
$$
If $v\in F_N \scseq(\t)_{\{I_{s}\}_{s \in S_{\c}}}$ then
the right hand side of the latter equation is $\geq -N$\,.
Thus the complex
$$
F_N \sc(\t)
$$
is concentrated in degrees
$
\geq -N\,.
$

\noindent
(3) Let us first consider the cochain complex
\begin{equation}
\label{gadget-F-N}
F_N\sc(\t) \,\Big/\, F_{N-1} \sc(\t)
\end{equation}
in the case when
$\t: S \to T$ is a usual $2$-tree.

If an element $v$ in (\ref{gadget}) represents
a non-zero vector in (\ref{gadget-F-N}) then
the set $\I_S$ (\ref{I-S01}) has exactly $N + |S|$ equivalence classes.
The total order on $\I_S$ gives a total order
on the set of these equivalence classes.
Hence the set of equivalence classes in $\I_S$ can be
identified with the ordinal $\{1,2, \dots, N + |S|\}$\,.
Furthermore each equivalence class is a subset
of $I_s$ for some $s\in S$\,.

Thus to every such element $v$ in (\ref{gadget})
we assign a surjection
\begin{equation}
\label{surjection}
\vs: \{1,2, \dots, N + |S|\} \to S
\end{equation}
from the ordinal $\{1,2, \dots, N + |S|\}$
to the set\footnote{Recall that $S$ is also equipped with
a total order but in general $\vs$ is not a map of ordinals.} $S$\,.

Not all such surjections can be gotten from the elements
of (\ref{gadget}) representing non-zero vectors
in (\ref{gadget-F-N}). The 2-tree $\t: S \to T$, the definition
of $\scseq$\,, and the definition of the elementary equivalence
relation impose the following conditions on the
possible surjections (\ref{surjection}):
\begin{enumerate}
\item[{\bf A}] $\vs(i)\neq \vs(i+1)$ $\forall ~~ i=1, 2, \dots,  N + |S|-1 $\,,

\item[{\bf B}] if $s\neq \ts$ and
$j_1 < i < j_2$ for $i\in \vs^{-1}(s)$ and
$j_1, j_2\in \vs^{-1}(\ts)$ then
$\t(s) < \t(\ts)$ in $T$\,,

\item[{\bf C}] if $\t(s)= \t(\ts)$ and $s< \ts$ then
all elements of $\vs^{-1}(s)$ are smaller than
all elements of $\vs^{-1}(\ts)$\,.

\end{enumerate}
Let us denote by $D(\t, N)$ the set of all surjections (\ref{surjection})
satisfying above conditions {\bf A},
{\bf B}, and {\bf C}.

It is not hard to see that the elements
of (\ref{gadget}) representing non-zero vectors
in (\ref{gadget-F-N}) and corresponding to the same
surjection (\ref{surjection}) span a subcomplex of
(\ref{gadget-F-N}). Furthermore for every map
(\ref{surjection}) this subcomplex is isomorphic to
the cochain complex $|\Xi_{N + |S|}|^{\bul+N}$\,,
where $|\Xi_k|^{\bul}$ are the complexes
described in the Appendix.

Thus (\ref{gadget-F-N})
is isomorphic to the direct sum of identical cochain complexes
\begin{equation}
\label{direct-sum}
F_N\sc(\t) \,\Big/\, F_{N-1} \sc(\t) \cong
\bigoplus_{\vs \in D(\t, N)} |\Xi_{N + |S|}|^{\bul+N}\,.
\end{equation}

Therefore, due to Proposition \ref{Xi-realize} from the Appendix we have,
\begin{equation}
\label{H-gadget-F-N}
H^{\bul} \Big( F_N\sc(\t) \,\Big/\, F_{N-1} \sc(\t) \Big) =
\begin{cases}
\displaystyle \bigoplus_{\vs \in D(\t, N)} \k \,, ~~ {\rm if} ~ \bul = -N\,, \\[0.5cm]
\phantom{aaaa} 0\,, \qquad {\rm otherwise}\,.
\end{cases}
\end{equation}

Let us now consider the cochain complex
\begin{equation}
\label{gadgetSC-F-N}
F_N\sc(\t) \,\Big/\, F_{N-1} \sc(\t)
\end{equation}
in the case when
$\t: S \to T$ is an SC $2$-tree.

As above $S_{\a}$ is the pre-image of the minimal element
of $T$ and $S_{\c} = S\, \setminus\, S_{\a}$\,.

If an element $v$ of (\ref{gadget-SC}) represents a non-zero
vector in (\ref{gadgetSC-F-N}) then the set
\begin{equation}
\label{I-S1}
\I_{S_{\c} \sqcup S_{\a}} = \bigsqcup_{s\in S_{\c}} I_s\, \sqcup \, S_{\a}
\end{equation}
has exactly $N + |S|$ equivalence classes.
The total order on $\I_{S_{\c} \sqcup S_{\a}}$ gives us a total order
on the set of its equivalence classes.
Hence the set of the equivalence classes can be identified
with the standard ordinal $\{1,2, \dots, N+ |S|\}$\,.
Furthermore, each equivalence class is either a subset
of $I_s$ for some $s\in S_{\c}$ or a one element
subset of $S_{\a}$\,. Thus we get a surjection
\begin{equation}
\label{surjection-SC}
\vs: \{1,2, \dots, N + |S|\} \to S
\end{equation}
from the ordinal $\{1,2, \dots, N + |S|\}$
to the set $S$\,.

As well as in the case of the usual 2-tree
this surjection satisfies above conditions {\bf A},
{\bf B}, and {\bf C}\,.

Let us remark that, since $S_{\a}$ is the pre-image of
the minimal element of $T$\,, conditions  {\bf A},
{\bf B}, and {\bf C} imposed on the surjection
(\ref{surjection-SC}) imply that for every $s\in S_{\a}$
the pre-image $\vs^{-1}(s)$ is a one element set.

As above we denote by $D(\t, N)$ the set of all
surjections (\ref{surjection-SC}) satisfying above
conditions  {\bf A}, {\bf B}, and {\bf C}.

Similarly to the case of a usual $2$-tree the set of
elements of (\ref{gadget-SC}) representing non-zero
vectors in (\ref{gadgetSC-F-N})
splits into the disjoint union of subsets, corresponding
surjections $\vs\in D(\t, N)$\,.
And similarly the elements
of (\ref{gadget-SC}) representing non-zero vectors
in (\ref{gadgetSC-F-N}) and corresponding to the same
map (\ref{surjection-SC}) span a subcomplex of
(\ref{gadgetSC-F-N}).
These subcomplexes are all isomorphic to the cochain complex
$$
|\Xi_{N+ |S_{\c}|}|^{\bul+N, 0}\,,
$$
where the bicomplexes $|\Xi_k|^{\bul, \bul} $ are described
in the Appendix.

It is not hard to see that the complex
$|\Xi_{N+ |S_{\c}|}|^{\bul, 0}$ consists of the field $\k$ placed in
degree $0$\,.

Thus for an SC $2$-tree $\t$ we have
\begin{equation}
\label{H-gadgetSC-F-N}
\Big( F_N\sc(\t) \,\Big/\, F_{N-1} \sc(\t) \Big)^{\bul}  \cong
\begin{cases}
\displaystyle \bigoplus_{\vs \in D(\t, N)} \k \,, ~~ {\rm if} ~ \bul = -N\,, \\[0.5cm]
\phantom{aaaa} 0\,, \qquad {\rm otherwise}
\end{cases}
\end{equation}
and statement $(3)$ holds in this case too.

\end{proof}
\subsection{Definition of the (SC) 2-operad $\scR$}
Using this filtration we give the following
definition.
\begin{Definition}
\label{defi-br}
We define the dg (SC) 2-operad $\scR$ as
a suboperad of $\sc$ with
\begin{equation}
\label{defi-br-eq}
\scR(\t) = \bigoplus_{N\ge 0} G^N \sc(\t)\,,
\end{equation}
where
$$
G^N \sc(\t) = \big\{ v \in F_N\sc(\t)^{-N} ~~\Big| ~~ d v \in
F_{N-1}\sc(\t) \big\}\,,
$$
and $\t$ is either a 2-tree or an SC 2-tree.
\end{Definition}

Lemma \ref{lmfil1} implies that the inclusion
$$
\scR \hookrightarrow \sc
$$
is a quasi-isomorphism.
Furthermore,
\begin{Proposition}
\label{scR-reduced}
The SC 2-operad $\scR$ is reduced.
\end{Proposition}

\begin{proof} Let $\t: S\to T$ be a 2-tree or an
SC 2-tree with $|S| \le 1$\,.
The condition $|S|\le 1$ implies that
the filtration $F$ on $\sc(\t)$ is trivial:
$F_{-1}\sc(\t)=0$ and
$F_N\sc(\t)=\sc(\t)$ for all $N\geq 0$.
Therefore, $\scR(\t)$ is simply the
vector space of degree $0$ cocycles in $\sc(\t)$
\begin{equation}
\label{for-small-tree}
\scR(\t)= \sc(\t)^0 \cap \ker d\,.
\end{equation}

Due to Lemma \ref{lmfil1} the complex
$\sc(\t)$ is concentrated in nonnegative
degrees. Hence
$$
H^0\big(\sc(\t)\big) = \sc(\t)^0 \cap \ker d\,.
$$

On the other hand, equations (\ref{H-gadget-F-N}) and
(\ref{H-gadgetSC-F-N}) imply that
$$
H^0\big(\sc(\t)\big) = \k [\,D(\t, 0)\,]
$$
and it is easy to see that if $|S|\le 1$ then
$D(\t, 0)$ is a one element set.

Thus $\scR(\t)$ is indeed isomorphic to $\k$\,.

It is not hard to check that the isomorphisms
$\k \cong \scR(\u_{\c})$ and $\k \cong \scR(\u_{\a})$
are given by the unit maps.
\end{proof}

\subsection{An increasing filtration on $\se$}
We will now define an analogue of the filtration $F$ from the previous subsection
for the (SC) operad $\se$.

Let us first consider the non-SC case.

Let $S$ be a finite set and $J, I_s,s\in S$ be non-empty finite ordinals. 
Every element
$$
u\in \se^{J}_{\{I_s\}_{s\in S}}
$$
gives us a total order $>_u$ on the set
$$
\I_S:=\bigsqcup\limits_{s\in S} I_s\,.
$$
Following Subsection \ref{ordeq} this order
gives us the elementary equivalence relation
on $\I_S$\,.
We denote the number of equivalence
classes in $\I_S$ by $|u|$
and define $F_N\se_{\{I_s\}_{s\in S}}^J$ as
the subset consisting of all elements
 $u\in \se_{\{I_s\}_{s\in S}}^J$ with $|u|\leq N+|S|$.

Let us now extend this definition for the SC-case. Let
$S$ be an SC  with $c_S= \a$ (the case $c_S=\c$ corresponds to the non-SC part and has  just
been considered).
We split $S$ as $S=S_\c\sqcup S_\a$ where
$S_{\c} = \chi^{-1} (\c)$ and $S_{\a} = \chi^{-1}(\a)$\,.

By definition an element
$$
u\in  \se(S_\c,S_\a)_{\{I_s\}_{s\in S_\c}}
$$
is a total order on the set
$$
\I_{S_{\c} \sqcup S_{\a}} =
\bigsqcup\limits_{s \in S_{\c}}
 I_{s} \,\, \sqcup \,\, S_{\a}
$$
subject to certain conditions.

Following Subsection \ref{ordeq-SC} this order
gives us the elementary equivalence relation on
$\I_{S_{\c} \sqcup S_{\a}}$\,.

Let us denote the number of
the equivalence classes in $\I_{S_{\c} \sqcup S_{\a}} $
by $|u|$ and
define
$$
F_N\se(S)_{\{I_s\}_{s\in S_\c}}
$$
as the set of all elements $u\in \se(S)_{\{I_s\}_{s\in S_\c}}$
with $|u|\leq  N+|S|$.

We claim that
\begin{Lemma} The filtration $F$ is
compatible with operadic compositions on
$\se$.
\end{Lemma}
\begin{proof} Similar to proof of Lemma
 \ref{lmfil}.
\end{proof}

This lemma implies that the filtration
$F$ on $\se$ is compatible with the
polysimpicial/cosimplicial structure.
Therefore, the formula
$$
F_N\, |\se|(S) = |F_N (\se)(S)|
$$
defines an increasing filtration on
the dg SC operad $|\se|$\,.

\begin{Lemma}
\label{lmfil-se}
The filtration $F$ on $|\se|$ satisfies
the following properties:

\begin{enumerate}

\item The operadic compositions in $|\se|$ are
compatible with the filtration $F$\,.

\item The complexes $F_N|\se|(S)$  are concentrated in the
degrees $\geq -N$\,.

\item The cohomology of the quotient
$F_N|\se|(S)/F_{N-1}|\se|(S)$ is concentrated
in the degree $-N$.
\end{enumerate}

\end{Lemma}

\begin{proof}
Since the proof is very similar to that of Lemma \ref{lmfil1}
we will only briefly outline the proof of (3).

Let  us first treat the non-SC part. Let $S$ be a finite set.
As well as for the  2-operad  $|\seq|$ the cochain
complex $F_N|\se|(S)/F_{N-1}|\se|(S)$ is isomorphic
to a direct sum of identical complexes
\begin{equation}
\label{direct-sum1}
F_N|\se|(S) \,\Big/\, F_{N-1} |\se|(S) \cong
\bigoplus_{\vs \in D(S, N)} |\Xi_{N + |S|}|^{\bul+N}\,,
\end{equation}
where the complexes $|\Xi_k|$ are described in the Appendix
and $D(S, N)$ is the set of surjections
\begin{equation}
\label{surjection-se}
\vs: \{1,2, \dots, N + |S|\} \to S
\end{equation}
satisfying the following conditions
\begin{enumerate}

\item[{\bf I}] $\vs(i)\neq \vs(i+1)$ $\forall ~~ i=1, 2, \dots,  N + |S|-1 $\,,

\item[{\bf II}] if $s\neq \ts \in S$ then it is impossible
to have $i_1, i_2 \in \vs^{-1}(s)$\,,
and $j_1, j_2\in \vs^{-1}(\ts)$ such that
$i_1 <j_1 < i_2 < j_2$\,.

\end{enumerate}
Thus Proposition \ref{Xi-realize} implies
statement (3) in the case $c_S =\c$\,.

Let us now consider the SC-case.  If $S$ is an SC set with $c_{S}= \a$\,,
$S_\c = \chi^{-1}(\c)$ and $S_\a = \chi^{-1}(\a)$
then the complex
 $F_N|\se|(S)/F_{N-1}|\se|(S)$ is isomorphic
to a direct sum of identical complexes
\begin{equation}
\label{direct-sum11}
F_N|\se|(S) \,\Big/\, F_{N-1} |\se|(S) \cong
\bigoplus_{\vs \in D(S, N)} |\Xi_{N + |S_{\c}|}|^{\bul+N, 0}\,,
\end{equation}
where the bicomplexes $|\Xi_k|^{\bul, \bul}$ are
described in the Appendix and
$D(S, N)$ is the set of surjections
(\ref{surjection-se}) satisfying above conditions
{\bf I}, {\bf II} and the additional condition:
\begin{enumerate}

\item[{\bf III}] if $s\in S_{\a}$ then $\vs^{-1}(s)$ consists of
exactly one element.

\end{enumerate}

Since the complex
$|\Xi_{N+ |S_{\c}|}|^{\bul, 0}$ consists of the field $\k$ placed in
degree $0$\,, statement (3) follows in this case too.

\end{proof}

We would like to remark that condition {\bf B} in the
proof of Lemma \ref{lmfil1} implies condition {\bf II}
in the proof of Lemma \ref{lmfil-se}.
Therefore, for every 2-tree $\t: S \to T$
we have the inclusion
\begin{equation}
\label{vazhnoe-incl}
D(\t, N) \subset D(S, N)\,.
\end{equation}

Similarly, if $\t$ is an SC 2-tree then conditions
{\bf A}, {\bf B}, and {\bf C} imply conditions {\bf I},
{\bf II}, and {\bf III}\,. Therefore, we have the inclusion
(\ref{vazhnoe-incl}) for SC 2-trees $\t$ as well.
We will use this inclusion later.

\subsection{Definition of the (SC) operad $\scoR$}
We now define a useful suboperad of $|\se|$

\begin{Definition}
\label{defi-braces}
We define the dg SC operad $\scoR$ as
a suboperad of $|\se|$ with
\begin{equation}
\label{defi-braces-eq}
\scoR(S) = \bigoplus_{N\ge 0} G^N |\se|(S)\,,
\end{equation}
where
$$
G^N |\se|(S) = \big\{ v \in F_N|\se|(S)^{-N} ~~\Big| ~~ d v \in
F_{N-1}|\se|(S) \big\}\,,
$$
and $S$ is an SC set.
\end{Definition}

Lemma \ref{lmfil-se} implies that the inclusion
$$
\scoR \hookrightarrow |\se|
$$
is a quasi-isomorphism.

\begin{Proposition}
\label{scoR-reduced}
The dg SC operad $\scoR$ is reduced.
\end{Proposition}
\begin{proof}
Let $S$ be an SC set with $|S|\le 1$\,.
It is not hard to construct a pruned 2-tree or a pruned SC 2-tree
$\t: S \to T$ with $S$ being the source ordinal.

It is easy to see that if $|S|<1$ then
$$
\scR(\t) = \scoR(S)
$$
as cochain complexes.

Thus the desired statement follows immediately
from Proposition \ref{scR-reduced}\,.
\end{proof}

Let us now consider a cofibrant resolution $\R\scR \to \scR$ of
$\scR$ in the closed model category of reduced dg (SC)
2-operads.

It is clear from the definitions of $\scR$ and $\scoR$ that
we have the embedding of dg (SC) 2-operads
$$
\scR \hookrightarrow \des\ \scoR\,.
$$
Since $\sym$ is the left adjoint functor for $\des$
this embedding produces the map
\begin{equation}
\label{glavnoe}
\sym\ \scR \to \scoR.
\end{equation}

Composing (\ref{glavnoe}) with the map
$$
\sym\ \R\scR \to \sym\ \scR
$$
we get the map
\begin{equation}
\label{Glavnoe}
\sym\ \R \scR \to \scoR.
\end{equation}

We claim that
\begin{Theorem}
\label{SymRbr-braces}
The map (\ref{Glavnoe}) is a quasi-isomorphism
of dg (SC) 2-operads.
\end{Theorem}
This theorem plays a crucial role in proving
our main result (Theorem \ref{th}).
We devote the next section to the proof of
this theorem.

\section{Proof of Theorem \ref{SymRbr-braces}}\label{proof1}
We need to show that for every (SC) set $S$ the map
\begin{equation}
\label{Glavnoe-S}
(\sym \R \scR) (S) \to \scoR (S)
\end{equation}
is a quasi-isomorphism of cochain complexes.

Due to the symmetrization formula
(see equation (\ref{sym-SC}))
$$
\sym\ \R \scR (S) = \colim_{\J(S)}  \R \scR_S\,.
$$
As $\R \scR$ is a cofibrant resolution
of $\scR$, Lemma \ref{scmodel} implies that the functor $\R \scR_S$ is cofibrant. Hence, the natural map
$$
\hocolim_{\J(S)}  \R \scR_S\to \colim_{\J(S)}  \R \scR_S 
$$
is a weak equivalence. Hence we have a zig-zag weak equivalence
$$
\sym\ \R \scR (S)  \stackrel\sim\to \hocolim_{\J(S)} \scR_S\,.
$$

Thus we need to show that the map
$$
\hocolim_{\J(S)} \scR_S \to \scoR (S)
$$
is a quasi-isomorphism of cochain complexes.

For this, it suffices to show that so is the map
\begin{equation}
\label{nado-qism}
\hocolim_{\J(S)} (F_N\scR /F_{N-1} \scR)_S  \to
 F_N \scoR (S)/ F_{N-1} \scoR (S)
\end{equation}
for every $N$\,.

Equations (\ref{H-gadget-F-N}), (\ref{H-gadgetSC-F-N}) and
statement (2) of Lemma \ref{lmfil1} imply that
for every 2-tree or SC 2-tree $\t$,
\begin{equation}
\label{Gr-scR}
F_N \scR(\t) \big/ F_{N-1} \scR(\t)
=\k[D(\t, N)][N]\,,
\end{equation}
where $\k[D(\t, N)][N]$ is considered as
a cochain complex with the zero differential.

Similarly, equations (\ref{direct-sum1}), (\ref{direct-sum11})
and statement (2) of Lemma \ref{lmfil-se} imply that
for every SC set $S$
\begin{equation}
\label{Gr-scoR}
F_N \scoR(S) \big/ F_{N-1} \scoR(S) = \k[D(S, N)][N]\,,
\end{equation}
where $\k[D(S, N)][N]$ is considered as
a cochain complex with the zero differential.

Let us recall that for every (SC) set $S$
and for every $\t \in \J(S)$ we have the
inclusion
$$
D(\t, N) \subset D(S, N)\,.
$$

Let $S$ be a finite  (SC) set.  Then for
$\vs \in D(S, N)$ we set $\J(\vs)\subset \J(S)$
to be the full subcategory of all $2$-trees 
$\t$ such that
$$
\vs \in D(\t, N)\,.
$$

Recall that for every (SC) set $S$ the category $\J(S)$
is a poset. It is not hard to see that for every morphism
$$
P: \t \to \wt
$$
in the category $\J(S)$ we have the inclusion
\begin{equation}
\label{inclusion-t-wt}
D(\wt, N) \subset D(\t, N)\,.
\end{equation}
Furthermore, the morphism
$$
F_N \scR(\wt) \big/ F_{N-1} \scR(\wt)
\to
F_N \scR(\t) \big/ F_{N-1} \scR(\t)
$$
corresponding to $P: \t \to \wt$ is given
by this inclusion.

Combining this observation with
equations (\ref{Gr-scR}) and (\ref{Gr-scoR})
we conclude that
\begin{equation}
\label{concern}
\hocolim_{\J(S)} (F_N\scR /F_{N-1} \scR)_S
= \bigoplus\limits_{\vs \in
D(S, N)} \hocolim_{\J(\vs)} \k\,,
\end{equation}
$$
(F_N/F_{N-1})\scoR(S)(N)=\bigoplus\limits_{\vs \in
D(S, N)} \k\,,
$$
and (\ref{nado-qism}) is induced by the
natural maps
\begin{equation}
\label{stya}
\hocolim_{\J(\vs)} \k\to \k\,,
\end{equation}
where, by abuse of notation, $\k$ denotes
both the underlying field and the functor
which assigns $\k$ to every object of $\J(\vs)$\,.

Thus it suffices to show that the map
(\ref{stya}) is a quasi-isomorphism for every
$\vs \in D(S, N)$\,.

The obvious topological counterpart of this
statement can be formulated as
\begin{Proposition}
\label{weak-equiv}
For every (SC) set $S$ and every element
$\vs \in D(S, N)$ the natural map
\begin{equation}
\label{styat}
\hocolim_{\J(\vs)} \pt \to \pt
\end{equation}
is a weak equivalence.
\end{Proposition}
In what follows, by abuse of notation, we denote a constant
functor from $\J(\vs)$ to another category by the underlying
object. For example, in (\ref{styat}) $\pt$ denotes both
the one-point space and the functor from $\J(\vs)$ to
the category of topological spaces which assigns $\pt$
to every object of $\J(\vs)$\,.

Let us postpone the proof of Proposition \ref{weak-equiv}
to the end of the section and show that this proposition
indeed implies that (\ref{stya}) is a quasi-isomorphism.

We, first, use the adjunction
\begin{equation}
\label{adjunction}
| ~ |_{\top}\, : \, \sSets \longleftrightarrow \Top
\, : \, C^{sing}_{*}
\end{equation}
between the category $\Top$ of topological spaces
and the category $\sSets$ of simplicial sets.
Here $|~|_{\top}$ denotes the realization functor and
$C^{sing}_{*}$ is the singular chain functor.

Using the fact that the adjunction (\ref{adjunction})
gives a Quillen equivalence between $\Top$ and
$\sSets$ it is not hard to deduce from Proposition
\ref{weak-equiv} its counterpart for simplicial sets.
Namely, Proposition \ref{weak-equiv} implies that
for every $\vs\in D(S, N)$ the natural map
\begin{equation}
\label{stya-simpl}
\hocolim_{\J(\vs)} \De^0 \to \De^0
\end{equation}
is a weak equivalence of simplicial sets, where
$$
\De^0  = \hom_{\D} ( ~ ,[0])
$$
is the terminal object of the category $\sSets$\,.

Therefore, for the simplicial Abelian group $\bbZ\De^0$\,,
the natural map
\begin{equation}
\label{stya-sAb}
\hocolim_{\J(\vs)} \bbZ\De^0 \to \bbZ\De^0
\end{equation}
is a weak equivalence.

Notice that, via the Dold-Kan correspondence,
(\ref{stya-sAb}) can be viewed as a map of
cochain\footnote{Here we reverse the standard grading of the Dold-Kan
correspondence.}
complexes
of Abelian groups. Furthermore, to say that (\ref{stya-sAb})
is a weak equivalence of simplicial Abelian groups is to say
that (\ref{stya-sAb}) is a quasi-isomorphism of the corresponding
cochain complexes.

Recall that the forgetful functor
$$
\Psi: \k - {\bf Vect} \to \Ab
$$
from the category $\k - {\bf Vect}$ of
$\k$-vector spaces to the category $\Ab$ of Abelian groups
admits the left adjoint functor
$$
\k\, \otimes_{\bbZ} \,\, : \Ab \to
\k - {\bf Vect} \,.
$$
Using this adjunction and the quasi-isomorphism
(\ref{stya-sAb}) we deduce that the natural map
\begin{equation}
\label{stya-k-mod}
\hocolim_{\J(\vs)} \k \De^0 \to \k \De^0
\end{equation}
is a quasi-isomorphism of cochain complexes of
$\k$-vector spaces. Here $\k \De^0$ is the cochain complex
$$
\dots
 \stackrel{{\rm id}}{\to}
\k
 \stackrel{0}{\to}
\k
\stackrel{{\rm id}}{\to} \k \stackrel{0}{\to} \k
$$
with the right most term placed in degree $0$\,.
This complex is obviously quasi-isomorphic to $\k$ placed
in degree $0$\,. And hence the map (\ref{stya}) is
indeed a quasi-isomorphism of cochain complexes.

In order to complete the proof of
Theorem \ref{SymRbr-braces} it remains to
prove Proposition \ref{weak-equiv}.

\subsection{Proof of Proposition \ref{weak-equiv}}
We need a cofibrant resolution of
the trivial functor from the poset $\J(\vs)$ to the category
of topological spaces.
The closed model structure on the category
of functors from $\J(\vs)$ to $\Top$ is obtained from that on
topological spaces using the transfer
principle\footnote{The transfer principle can
be applied in this case because $\J(\vs)$ is
a finite poset and the Quillen's path-object argument
obviously works for topological spaces.}
of C. Berger and I. Moerdijk \cite{BM}.
In other words, fibrations (resp. weak equivalences)
between functors from $\J(\vs)$ are object-wise fibrations
(resp. object-wise weak equivalences).

In order to construct the resolution, given a finite set $S$,
 we consider the configuration space
$\conf(S)$ of distinct points on $\bbR^2$ labeled by
elements of  $S$\,.

It is known that
the space $\conf(S)$ admits a cellular subdivision into
the Fox-Neuwirth cells
\cite{Berger}, \cite{GJ}, \cite{Sasha1}. Each Fox-Neuwirth
cell $\FN_{\t}$ corresponds to a pruned 2-tree $\t:S \to T$ and it can
be defined as the space of all injective maps from the 2-tree $\t$
to the generalized 2-tree:
$$
(x,y) \to x :
\bbR^2 \to \bbR\,,
$$
where on $\bbR^2$ we use the lexicographic order.

In other words, a configuration $\{(x_s, y_s)\}_{s\in S}$
belongs to $\FN_{\t}$ iff the following conditions
are satisfied:

--- if $\t(s)= \t(\ts)$ and $s< \ts$
then $x_{s}= x_{\ts}$ and $y_s < y_{\ts}$\,,

--- if $\t(s)< \t(\ts)$ then $x_{s}< x_{\ts}$\,.

An example of a configuration from $\FN_{\t_1}$ for the
$2$-tree
$$
\t_1: \{1,2,3,4,5\} \to \{1,2,3\}
$$
$$
\t_1 (1)= \t_1(2) =1\,, \qquad \t_1 (3)= \t_1(4) =2\,, \qquad
\t_1(5) = 3
$$
is depicted in figure \ref{FN1}
\begin{figure}[htb]
\psfrag{1}[]{$1$}
\psfrag{2}[]{$2$}
\psfrag{3}[]{$3$}
\psfrag{4}[]{$4$}
\psfrag{5}[]{$5$}
\includegraphics[width=5cm,height=5cm]{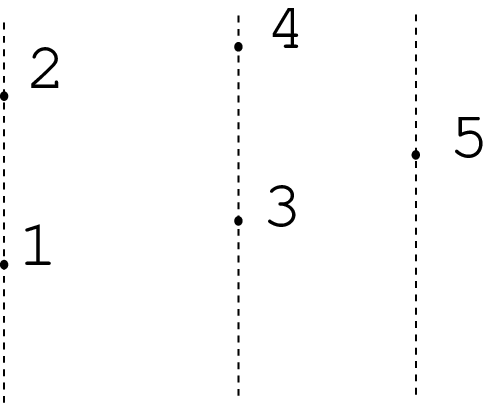}
\caption{A typical point of $\FN_{\t_1}$} \label{FN1}
\end{figure}

This construction can be easily generalized
to pruned SC 2-trees. Namely, if $\t: S \to T$ is a pruned SC 2-tree
with $S_{\a}$ being the preimage of the minimal element
of $T$ and
$S_{\c}= S \setminus S_{\a}$ then $\FN_{\t}$ consists
of configurations $\{(x_s, y_s)\}_{s\in S}$
satisfying the following conditions:

--- if $s\in S_{\a}$ then $x_{s}=0$\,; if $s\in S_{\c}$ then $x_s > 0$\,,

--- if $\t(s)= \t(\ts)$ and $s< \ts$
then $x_{s}= x_{\ts}$ and $y_s < y_{\ts}$\,,

--- if $\t(s)< \t(\ts)$ then $x_{s}< x_{\ts}$\,.

Recall that for pruned SC 2-trees the range $\t(S)$
does not in general include the minimal element.
In other words, the subset $S_{\a}$ may be empty.
In this case we still require that $x_s > 0$ for
$s\in S_{\c}$\,.

If $S$ is an (SC) set then
for every map $P: \t \to \wt$ of pruned (SC) 2-trees in
the category $\J(S)$ we have the obvious inclusion
\begin{equation}
\label{FN-t-FN-wt}
\FN_{\wt} \hookrightarrow \pa \FN_{\t}\,,
\end{equation}
where $\pa \FN_{\t}$ denotes the boundary of
the Fox-Neuwirth cell $\FN_{\t}$\,.

For example, we may consider the 2-tree
$$
\t_2: \{1,2,3,4,5\} \to \{1,2\}
$$
$$
\t_2 (1)= \t_2(2) =1\,, \qquad \t_2 (3)= \t_2(4) = \t_2(5) = 2
$$
with a (unique) map in $\J(\{1,2,3,4,5\})$
$$
P: \t_1 \to \t_2\,,
$$
$$
P_S = id\,, \qquad P_T(1) = 1\,, \qquad P_T(2)=P_T(3) = 2\,.
$$
A configurations from $FN_{\t_2}$ consists of a pair of
distinct vertical lines; the left line carries points
$1$ and $2$ such that $1$ is below $2$; the right line
carries points $3$, $4$, $5$ which are put in the order
from the bottom to the top.  (See figure \ref{FN2}.)
\begin{figure}[htb]
\psfrag{1}[]{$1$}
\psfrag{2}[]{$2$}
\psfrag{3}[]{$3$}
\psfrag{4}[]{$4$}
\psfrag{5}[]{$5$}
\includegraphics[width=5cm,height=5cm]{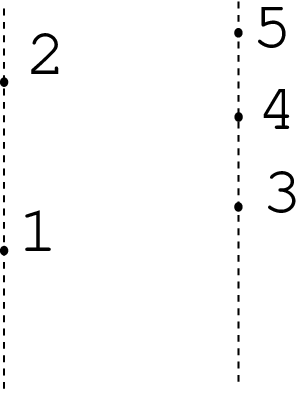}
\caption{A typical point of $\FN_{\t_2}$} \label{FN2}
\end{figure}
It is clear that $\FN_{\t_2}$ belongs to the boundary of
$\FN_{\t_1}$\,.

Let $\vs\in D(S, N)$ and $\J(\vs)$ be the sub-poset
of $\J(S)$ defined above. Using the inclusion
(\ref{FN-t-FN-wt}) we upgrade the correspondence
\begin{equation}
\label{Phi-vs}
\t \to \Phi_{\vs}(\t) =
\bigcup_{\wt\in \J(\vs); ~~ \t\to \wt}
\FN_{\wt}
\end{equation}
to the functor
$$
\Phi_{\vs}  : \J(\vs) \to \Top\,.
$$
The union in (\ref{Phi-vs}) is taken over all
the pruned (SC) $2$-trees
$\wt\in \J(\vs)$ for which we have a map
from $\t$ to $\wt$\,.

\begin{Example}
\label{primer-al-be-ga-de}
{\rm
We consider $S= \{\al,\beta,\ga, \de\}$ with $c_{S}= \a$\,,
$\chi(\al)= \chi(\ga)= \chi(\de) =\c$\,, $\chi(\beta)= \a$\,, and
$\vs$ being the following map
$$
\vs : \{1,2,3,4, 5,6\} \to S
$$
$$
\vs(1)=\al\,, \qquad \vs(2) =\de \,, \qquad \vs(3) = \ga\,,
\qquad \vs(4) = \beta\,, \qquad \vs(5)= \ga, \qquad
\vs(6)= \de\,.
$$
The map $\vs$ is an element of $D(S, 2)$ and
the SC $2$-tree $\t : \{\beta<\ga<\al<\de\} \to \{1,2,3,4\}$
$$
\t(\beta)=1, \qquad \t(\ga)=2, \qquad \t(\al)=3, \qquad \t(\de) = 4
$$
is an object of $\J(\vs)$\,.

There are exactly three pruned SC 2-trees
$\wt\in \J(\vs)$ for which there is a map
$\t \to \wt$\,. The first one is $\wt_1 = \t$ and the second
one is $\wt_2 :  \{\beta<\ga< \al<\de\} \to \{1,2,3\}$
$$
\wt_2(\beta)=1, \qquad \wt_2(\ga)=2, \qquad \wt_2(\al)=\wt_2(\de)=3\,.
$$
The third SC 2-tree $\wt_3 : \{\beta<\al< \ga<\de\} \to \{1,2,3\}$
$$
\wt_3(\beta)=1, \qquad \wt_3(\al)=\wt_3(\ga)=2, \qquad \wt_3(\de)=3\,.
$$

So the space $\Phi_{\vs}(\t)$ consists of
configurations $\{(x_s, y_s)\}_{s\in \{\al,\beta,\ga, \de\}}$
satisfying the following conditions:

--- $x_{\beta}=0 < x_{\ga} \le x_{\al} \le x_{\de}$\,, and $x_{\ga}<x_{\de}$\,,

--- if $x_\al = x_\ga$ then $y_{\al} < y_{\ga}$\,,

--- if $x_\al = x_\de$ then $y_{\al} < y_{\de}$\,.
}
\end{Example}

\begin{Proposition}
\label{cofib-resol}
Let $S$ be an (SC) set and $\vs\in D(S, N)$\,.
Then the functor $\Phi_{\vs}$ (\ref{Phi-vs}) is
a cofibrant resolution of the trivial functor
from $\J(\vs)$ to the category of topological
spaces.
\end{Proposition}

\begin{proof}
Let $S$ be an (SC) set and
$\vs \in D(S, N)$\,. Let us show that
$\Phi_{\vs}(\t)$ is contractible for every pruned
(SC) $2$-tree $\t: S\to T$ for which
$\vs \in D(\t, N)$\,.

We give a detailed proof of contractibility of
$\Phi_{\vs}(\t)$ in the case when $c_{S}= \c$ (i.e. $S$ is a usual, non-SC, set) and
hence $\t$ is a pruned (non-SC) 2-tree.
The SC case $c_S=\a$ is very similar.

The 2-tree $\t: S\to T$ gives us a total order
on the set $S$. So we identify $S$ with the ordinal
$\{1,2, 3, \dots, |S|\}$ and denote by
$(x_i, y_i)$ the coordinates of the point labeled by
$i \in \{ 1,2,3, \dots, |S| \}$\,.

Next, we consider the following sequence of subspaces
$$
\Phi_{\vs}(\t)= F_0 \supset F_1 \supset F_2 \supset \dots \supset F_{|S|}
$$
where $F_k$ consists of configurations $\{(x_i, y_i)\} \in \Phi_{\vs}(\t)$
with
$$
y_i = i\,, \qquad \forall ~~i \le k\,.
$$
Let us show that $F_{k+1}$ is a deformation
retract of $F_k$ for all $k =0,1,2, \dots, |S|-1$\,.

A deformation retraction $f: F_k \times [0,1] \to F_k $
of $F_k$ onto $F_{k+1}$ is given by the formula:
\begin{equation}
\label{f-t}
f(\{(x_i, y_i)\}, t) =
\{(x_i, y_i(t) )\}\,,
\end{equation}
where
$$
y_i(t) =
\begin{cases}
i \,, \qquad  {\rm if} \qquad i \le k\,, \cr
(1-t) y_i + t(k+1 + y_i-y_{k+1}) \,, ~ {\rm if} ~ i > k \,.
\end{cases}
$$
We need to show that for all $t \in [0,1]$ and
for all configurations $\{(x_i, y_i)\}\in F_k$
the point
$f(\{(x_i, y_i)\}, t) $ belongs to $\Phi_{\vs}(\t)$\,. More precisely,
we need to check that if $x_i= x_j$
and $i<j$ then $y_i(t) < y_j(t)$ for all $t \in [0,1]$\,.

First, it is obvious that if $i<j\le k $ the
$y_i(t)< y_j(t)$ regardless of whether $x_i$ equals $x_j$
or not. Second, it is not hard to see that if
$k < i< j$ and $y_i < y_j$ then
$y_i(t) < y_j(t)$ for all $t \in [0,1]$\,.
Finally, if $i\le k < j$ and $x_i = x_j$ then
the configuration $\{(x_i, y_j)\}$ belongs to the Fox-Neuwirth cell
$\FN_{\wt}$ corresponding to  a pruned 2-tree $\wt$ for which
$$
\wt(i)= \wt(j)\,.
$$
The latter implies that $\wt(i)= \wt(i+1)= \dots = \wt(j-1) = \wt(j)$
and hence
$$
x_i= x_{i+1} = \dots = x_{j-1} = x_{j}\,.
$$
Therefore $y_i< y_{i+1}< \dots < y_{j-1} < y_j$ and,
in particular\footnote{In this case $y_j= y_{k+1}$ only if $j=k+1$},
$$
y_j \ge y_{k+1} > y_k = k\,.
$$
Using these inequalities we conclude that
for all $t\in [0,1]$
$$
y_j(t) = (1-t)y_j + t(k+1) + t(y_j- y_{k+1}) >
k + t (y_j - y_{k+1}) \ge k\,.
$$
On the other hand $y_i(t) \le k$\,.
Thus, if $x_i = x_j$ and $i < j$ then
$$
y_j(t) > y_i(t)
$$
for all $t\in [0,1]$\,.

Furthermore, if $y_i = i$ for all $i\le k+1$
then $y_i(t)\equiv i$ for all $i\le k+1$\,.
Thus $f$ is indeed a deformation retraction of $F_k$
onto $F_{k+1}$\,.

Let us now identify $T$ with the standard ordinal
$\{1,2,3, \dots, |T|\}$\,. Next we note
that if $\wt\in \J(\vs)$ admits a map $\t \to \wt$
then equality $\t(i)= \t(j)$ implies the equality
$\wt(i) = \wt(j)$\,. Hence, if $\t(i) = \t(j)$ then
$x_i= x_j$ for every configuration
$\{(x_i, y_i)\}\in \Phi_{\vs}(\t)$\,.

Therefore the function $i \to x_i$ factors through
$$
\t: \{1,2,3, \dots, |S|\} \to \{1,2,3, \dots, |T|\}
$$
and hence, we may describe configurations
from $\Phi_{\vs}(\t)$ using the collections of coordinates
$\{z_l, y_i\}$\,, $z_l, y_i\in \bbR$  where
$l\in  \{1,2,3, \dots, |T|\}$ and $i\in  \{1,2,3, \dots, |S|\}$\,.

For every configuration $\{z_l, y_i\}$ from $\Phi_{\vs}(\t)$
we have
\begin{equation}
\label{inequal}
z_1 \le z_2 \le z_3 \le \dots \le z_{|T|}
\end{equation}
and if $z_l= z_m$ for $l\neq m$ then the
corresponding configuration belongs to the Fox-Neuwirth
cell $\FN_{\wt}$ of a 2-tree $\wt \neq \t$\,.

To show the contractibility of $F_{|S|}$ we
consider the following
sequence of subspaces:
$$
F_{|S|} = G_0 \supset G_1 \supset G_2 \supset \dots \supset G_{|T|} \cong \pt.
$$
where $G_k$ consists of configurations $\{z_l, y_i\}\in F_{|S|}$
satisfying the following condition
$$
z_l = l \,, \qquad \forall \quad l \le k\,.
$$
In terms of the original coordinates $(x_i, y_i)$
the latter condition reads
$$
x_i = \t(i)\,, \qquad {\rm if} \quad  \t(i) \le k\,.
$$

We show that for all $k \le |T|-1$ the space $G_{k+1}$ is a deformation
retract of $G_k$\,.

The desired deformation retraction is defined by the
formula
\begin{equation}
\label{g-t}
g(\{z_l, y_i\}, t) =
\{z_l(t), y_i\}\,,
\end{equation}
where
$$
z_l(t) =
\begin{cases}
l \,, \qquad  {\rm if} \qquad l \le k\,, \cr
(1-t) z_l + t(k+1 + z_l- z_{k+1}) \,, ~ {\rm if} ~ l > k \,.
\end{cases}
$$

To prove that the configuration $\{z_l(t), y_i\}$ belongs
to $F_{|S|}$ for all $t\in [0,1]$ we need to check that
inequalities
\begin{equation}
\label{inequal1}
z_1(t) \le z_2(t) \le z_3(t) \le \dots \le z_{|T|}(t)
\end{equation}
hold for all $t\in [0,1]$\,. Furthermore we need to
check that if $z_l< z_m$ then $z_l(t) < z_m(t)$ for all
$t\in [0,1]$\,.

In the case $l < m \le k$ we simply have the
inequality $z_l(t) < z_m(t)$\,.
Also it is not hard to see that in the case
$k < l < m$ the inequality $z_{l}(t)\le z_m(t)$
(resp. $z_l(t)< z_m(t)$) follows
from $z_l \le z_m$ (resp. $z_l < z_m$)\,.

Thus it remains to consider the case $l=k$ and
$m=k+1$\,.

In this case we have $z_k(t) \equiv z_k = k$\,.
Furthermore, due to (\ref{inequal}) we have
$z_{k+1} \ge k$ and hence
$$
z_{k+1}(t) = (1-t)z_{k+1} + t(k+1) \ge (1-t) k + t k =k\,.
$$
It is also obvious that if $z_{k+1} > k$ then
$$
z_{k+1}(t) = (1-t)z_{k+1} + t(k+1) > (1-t) k + t k = k =z_k(t)
$$
for all $t\in [0,1]$\,.

Finally, it is clear that if $z_{k+1}= k+1$ then
$$
z_{k+1}(t)\equiv k+1\,.
$$
Thus $g$ (\ref{g-t}) is indeed a deformation retraction
of $G_k$ onto $G_{k+1}$\,.

Since $G_{|T|}$ is a one-point space we conclude that
$F_{|S|}$, and hence, the space $\Phi_{\vs}(\t)$ is contractible.

The proof of the fact that $\Phi_{\vs}$ is a cofibrant object
in the category of functors from $\J(\vs)$ to $\Top$ is
very similar to the proof of Theorem 7.2 from \cite{Bat}.

Following the arguments of \cite{Bat} we define the following
sequence of functors
$$
\Phi^m_{\vs}\,, \qquad  m\in \bbZ\,, \qquad m\ge 0\,.
$$
On the level of objects the functor $\Phi^m_{\vs}$
operates as
\begin{equation}
\label{Phi-vs-m}
\Phi^m_{\vs} (\t)
=
\begin{cases}
\Phi_{\vs}(\t) \,, \qquad  {\rm if} \quad |S|+ |T| < m\,,
\quad {\rm and}\quad \t \quad {\rm is~ a~ 2-tree}, \cr
\Phi_{\vs}(\t) \,, \qquad  {\rm if} \quad |S|+ |T|-1 < m\,,
\quad {\rm and}\quad  \t \quad {\rm is~ an~SC~ 2-tree}, \cr
\emptyset \,, \qquad {\rm otherwise}.
\end{cases}
\end{equation}
We would like to remark that the number $|S|+|T|$ (resp. $|S|+|T|-1$)
for a 2-tree $\t:S\to T$ (resp. for
an SC 2-tree $\t:S\to T$) is the dimension of the Fox-Neuwirth cell
$\FN_{\t}$\,. Thus the collection $\Phi^m_{\vs}$ may be considered
as a filtration of $\Phi_{\vs}$ by dimension.

We have the obvious sequence of natural transformations
$$
\Phi^0_{\vs} \to \Phi^1_{\vs} \to \Phi^2_{\vs} \to \dots
$$
and the functor $\Phi_{\vs}$ is the sequential colimit
$$
\Phi_{\vs} = \colim_{m} \Phi^m_{\vs}\,.
$$

Similarly to the proof of Theorem 7.2 from \cite{Bat}
we show that for every $m$ the natural transformation
$$
\Phi^m_{\vs} \to \Phi^{m+1}_{\vs}
$$
is a cellular extension generated by a cofibration.

Thus $\Phi_{\vs}$ is indeed a cofibrant object in
the category of functors from $\J(\vs)$ to the category
of topological spaces.

This completes the proof of Proposition
\ref{cofib-resol}.

\end{proof}

Now that we have a cofibrant resolution $\Phi_{\vs}$ of the
trivial functor from $\J(\vs)$ to $\Top$ we
prove Proposition \ref{weak-equiv} by showing that
the space
\begin{equation}
\label{X-vs}
X_{\vs} = \colim_{\J(\vs)} \Phi_{\vs}
\end{equation}
is contractible for every surjection $\vs\in D(S, N)$\,,
where $S$ is an (SC) set.

It is easy to see that
\begin{equation}
\label{X-vs1}
X_{\vs} = \bigcup_{\t\in \J(\vs)} \FN_{\t}\,.
\end{equation}

To get a more explicit description of the space $X_{\vs}$
(\ref{X-vs1}) we recall that the set $D(S, N)$ consists of
surjections
$$
\vs : \{1,2,3, \dots |S|+N\} \to S
$$
from the standard ordinal $\{1,2,3, \dots |S|+N\}$ to
the set $S$\,; the surjections $\vs$ should satisfy two
conditions {\bf I} and {\bf II} from the proof of
Lemma \ref{lmfil-se}; if, in addition $c_S= \a$, then
we should also impose on $\vs$ condition {\bf III}
from the proof of the same lemma.

Let us also recall that a 2-tree (an SC 2-tree)
$\t: S \to T$ belongs
to $\J(\vs)$ iff the following conditions are met:

--- if for $s\neq \ts$ there exist $i_1, i_2\in \vs^{-1}(\ts)$ and $i\in \vs^{-1}(s)$
such that $i_1 < i < i_2$ then $\t(s)< \t(\ts)$ in the (SC) ordinal $T$\,,

--- if $\t(s) = \t(\ts)$ and  $s <_{\t} \ts$ then
all elements of $\vs^{-1}(s)$
are smaller than all elements of $\vs^{-1}(\ts)$\,.

Here $<_{\t}$ is the total order on $S$ coming from
the structure of the (SC) 2-tree $\t$\,.

Thus the space $X_{\vs}$ (\ref{X-vs1}) consists of the
configurations $\{(x_s, y_s)\}$ from $\conf(S)$ satisfying
the following conditions:
\begin{enumerate}
\item[{\bf C1}] if $\exists$ $i_1, i_2\in \vs^{-1}(\ts)$ and $i\in \vs^{-1}(s)$
such that $i_1 < i < i_2$ then $x_{s} < x_{\ts}$

\item[{\bf C2}] if $x_{s}= x_{\ts}$ and all elements of $\vs^{-1}(s)$
are smaller than all elements of $\vs^{-1}(\ts)$ then
$y_s < y_{\ts}$\,.
\end{enumerate}

If $c_{S} = \a$ then we have to impose on the configuration
$\{(x_s, y_s)\}$ the additional condition
\begin{enumerate}

\item[{\bf C3}] if $\chi(s)= \a$ then $x_s = 0$ and if
$\chi(s)= \c$ then $x_{s} > 0$\,.

\end{enumerate}

\noindent\\
{\bf Remark.} Let $S$ be a usual (non-SC) set.
It can be shown that every surjection $\vs\in D(S, N)$
gives us a pair of complementary orders on the set $S$
in the sense of M. Kontsevich and Y. Soibelman \cite{KS}.
(See also Section 2 in \cite{Bat} about complementary
orders and higher trees.) To a pair of complementary
orders $>_0$ and $>_1$ M. Kontsevich and Y. Soibelman
assign a subspace $X_{>_0, >_1}$ \cite{KS} of the compactified
configuration space of points on $\bbR^2$ labeled by elements of $S$\,.
Our space $X_{\vs}$ is an uncompactified version
of the subspace considered by M. Kontsevich and Y. Soibelman
in \cite{KS}.

\subsubsection{Contractibility of $X_{\vs}$}
We give a detailed proof of the contractibility of $X_{\vs}$ (\ref{X-vs1})
in the SC case when the color $c_{S}$ of the SC set $S$ is $\a$\,.
The non SC case ($c_{S}=\c$) is very similar.

Although every SC 2-tree $\t\in \J(\vs)$ gives us a
total order $>_{\t}$ on $S$, we equip the set $S$ with yet another total order
which we denote by $<_{\vs}$\,. Namely, we set $s <_{\vs} \ts$ iff

--- either $\exists$ $i_1, i_2\in \vs^{-1}(\ts)$ and $i\in \vs^{-1}(s)$
such that $i_1 < i < i_2$ or

--- all elements of $\vs^{-1}(s)$ are smaller than all elements
of $\vs^{-1}(\ts)$\,.\\

\noindent
{\bf Warning.} In general, the order total $>_{\t}$ on $S$ coming
from the structure of an SC 2-tree $\t\in \J(\vs)$
does not coincide with the order $>_{\vs}$\,.
Thus, in Example \ref{primer-al-be-ga-de}, the map
$\vs$ induces on the SC set $S$ the order
$$
\al<\beta< \ga< \de\,.
$$
On the other hand we have a pruned SC 2-tree
$\t : \{\beta<\ga<\al<\de\} \to \{1,2,3,4\}$
which belongs to $\J(\vs)$\,. A similar example
can be found for an SC set $S$ with $c_{S}=\c$\,.\\

Using the total order $>_{\vs}$ we identify $S$ with the standard ordinal
$\{1 < 2 < 3 < \dots < |S|\}$\,.

Next we define the following functions on $\conf(S)$
\begin{equation}
\label{mu}
\mu_k(\{ (x_s, y_s)\}) = \min (y_k, y_{k+1}, \dots, y_{|S|})
\end{equation}
which are obviously continuous.

Then we introduce the sequence of subspaces
$$
X_{\vs} = Y_0 \supset Y_1 \supset \dots \supset Y_{|S|}\,,
$$
where $Y_k$ consists of configurations $\{(x_s, y_s)\}\in X_{\vs}$
satisfying the properties
\begin{equation}
\label{do-k}
y_s = y_1 + s-1\,, \qquad \forall \quad  s\le k\,,
\end{equation}
\begin{equation}
\label{posle-k}
\mu_{k+1}(\{ (x_s, y_s)\})= y_k+1\,.
\end{equation}

Let us show that $Y_{k+1}$ is homotopy equivalent to
$Y_{k}$ for all $k < |S|$\,.

For this purpose we introduce an intermediate
subspace $Z_k$
$$
Y_k\supset Z_k \supset Y_{k+1}\,.
$$
This subspace consists of configurations $\{(x_s, y_s)\}\in Y_k$
satisfying the property
\begin{equation}
\label{y-mu}
y_{k+1} = \mu_{k+1}(\{ (x_s, y_s)\})\,.
\end{equation}

Let us consider the map
$h: Y_k \times [0,1] \to Y_k $
\begin{equation}
\label{h-t}
h(\{(x_s, y_s)\}, t) =
\{(x_s, y_s(t) )\}\,,
\end{equation}
where
$$
y_s(t) =
\begin{cases}
y_s \,, \qquad  {\rm if} \qquad s \neq k+1\,, \cr
(1-t) y_{k+1} + t \mu_{k+1}(\{ (x_s, y_s)\}) \,, ~ {\rm if} ~ s = k+1 \,.
\end{cases}
$$
In order to show that $h(\{(x_s, y_s)\}, t) \in Y_k$
we only need to check condition {\bf C2} for all $t\in [0,1]$\,.

It is clear that
\begin{equation}
\label{ineq-kplus1}
y_{k+1} \ge y_{k+1}(t) \ge \mu_{k+1}(\{ (x_s, y_s)\})\,,
\qquad \forall \quad  t\in [0,1]\,.
\end{equation}

Since $\{(x_s, y_s)\}\in Y_k$ we have
$$
\mu_{k+1}(\{ (x_s, y_s)\}) >y_s\,, \qquad \forall \quad s \le k
$$
and hence
$$
y_{k+1}(t) > y_s\,, \qquad \forall \quad s \le k, \quad t\in [0,1]\,.
$$

Furthermore, since condition {\bf C2} is satisfied for $\{(x_s, y_s)\}$
we conclude that all points $(x_s, y_s)$ with $s>k+1$ and
$x_s=x_{k+1}$ lie above the point $(x_{k+1}, y_{k+1})$\,.
Combining this observation with inequality (\ref{ineq-kplus1})
we conclude that if $s > k+1$ and $x_s = x_{k+1}$ then
$y_{k+1}(t)< y_s$ for all $t\in [0,1]$\,.

It is clear that $h(\{(x_s, y_s)\}, 1)\in Z_k $ and
for all $\{(x_s, y_s)\}\in Z_k $ we have
$$
h(\{(x_s, y_s)\}, t) = \{(x_s, y_s)\}\,, \qquad \forall \quad t\in [0,1]\,.
$$
Thus $h$ is a deformation retraction of $Y_k$
onto $Z_k$\,.

It is clear that the subspace $Y_{k+1}$ consists
of configurations $\{(x_s, y_s)\}\in Z_k$ satisfying
the additional property
$$
\mu_{k+2}(\{(x_s, y_s)\}) = y_{k+1}+1\,.
$$
So we consider the map
$h_Z: Z_k \times [0,1] \to Z_k $
\begin{equation}
\label{h-t-Z}
h_Z(\{(x_s, y_s)\}, t) =
\{(x_s, y_s(t) )\}\,,
\end{equation}
where
$$
y_s(t) =
\begin{cases}
y_s \,, \qquad  {\rm if} \qquad s \le k+1\,, \cr
y_s + t\big(y_{k+1}+1 - \mu_{k+2}(\{ (x_s, y_s)\})\, \big)\,,
\qquad {\rm if} \qquad s > k+1\,.
\end{cases}
$$
In order to show that $h_Z$ lands in $Z_k$
we need to check condition {\bf C2} and condition (\ref{y-mu}).

Since
$$
\min (y_{k+2}(t), y_{k+3}(t), \dots, y_{|S|}(t)) =
$$
$$
\min (y_{k+2}, y_{k+3}, \dots, y_{|S|}) +
t\big(y_{k+1}+1 - \mu_{k+2}(\{ (x_s, y_s)\})\, \big) =
$$
$$
(1-t)\mu_{k+2}(\{ (x_s, y_s)\}) + t(y_{k+1}+1) \ge \mu_{k+1}(\{ (x_s, y_s)\})
$$
we conclude that
$$
\mu_{k+1}(\{ (x_s, y_s(t))\})
$$
does not depend on $t$\,. Thus condition (\ref{y-mu}) is satisfied.

Next, if $s \ge k+2$ then
$$
y_s(t) \ge \mu_{k+2}(\{ (x_s, y_s)\})  +
t\big(y_{k+1}+1 - \mu_{k+2}(\{ (x_s, y_s)\})\, \big) =
$$
$$
(1-t)\mu_{k+2}(\{ (x_s, y_s)\}) + t(y_{k+1}+1) > y_{k+1}
$$
for all $t\in (0,1]$ because $\mu_{k+2}(\{ (x_s, y_s)\}) \ge y_{k+1}$
and $y_{k+1}+1 > y_{k+1}$\,. Hence
$$
y_s(t) > y_{\ts}
$$
for all $s\ge k+2$, $\ts \le k+1$ and $t\in (0,1]$\,.

Furthermore, if for $s,\ts\ge k+2$ we have $y_s> y_{\ts}$
then obviously $y_s(t) > y_{\ts}(t)$ for all $t\in [0,1]$\,.
Thus we conclude that condition {\bf C2} is satisfied
for every configuration $h_Z(\{(x_s, y_s)\}, t)$\,.

It is not hard to see that for all $\{(x_s, y_s)\}\in Z_k$
$$
h_Z(\{(x_s, y_s)\}, 1)\in Y_{k+1}
$$
and for all $t\in [0,1]$ and $\{(x_s, y_s)\}\in Y_{k+1}$
$$
h_Z(\{(x_s, y_s)\}, t) = \{(x_s, y_s)\}\,.
$$
Thus $h_Z$ is a deformation retraction of $Z_k$ onto
$Y_{k+1}$.

We proved that $X_{\vs}$ is homotopy equivalent to
the subspace $Y_{|S|}$ which consists of configurations
$\{(x_s, y_s)\}\in X_{\vs}$ satisfying the property
\begin{equation}
\label{do-k-posle}
y_s = y_1 + s-1\,, \qquad \forall \quad  s\in S\,.
\end{equation}

To show that $Y_{|S|}$ is contractible we
set, as above, $S_{\a} = \chi^{-1}(\a)$ and
$S_{\c}= \chi^{-1}(\c)$\,.

Due to Condition {\bf C3} $x_{s}=0$ for all $s\in S_{\a}$
and $x_{s} >0 $ for all $s\in S_{\c}$\,.

Restricting the total order $>_{\vs}$ from $S$ to $S_{\c}$
we get an isomorphism
$$
\beta : S_{\c} \to \{1<2<3< \dots < |S_{\c}|\}
$$
from $S_{\c}$ to the standard ordinal $\{1<2<3< \dots < |S_{\c}|\}$\,.

Using this isomorphism we define the following map
$H: Y_{|S|}\times [0,1] \to Y_{|S|}$
\begin{equation}
\label{H-t}
H(\{(x_s, y_s)\}, t) =
\{(x_s(t), y_s)\}\,,
\end{equation}
where
$$
x_s(t) =
\begin{cases}
0 \,, \qquad  {\rm if} \qquad s\in S_{\a}\,, \cr
(1-t) x_s + t \beta(s) \,, ~ {\rm if} ~ s\in S_{\c} \,.
\end{cases}
$$

Let us show that $H$ indeed lands in $X_{\vs}$\,.

Since $x_s(t)>0$ for all $s\in S_{\c}$ and $t\in [0,1]$
we need to check Condition {\bf C1} only for $s, \ts\in S_{\c}$\,.

If $s, \ts\in S_{\c}$, $s\neq \ts$ and
there exists $i_1, i_2\in \vs^{-1}(\ts)$ and $i\in \vs^{-1}(s)$
such that $i_1 < i < i_2$ then $x_{s} < x_{\ts}$ and $\beta(s)< \beta(\ts)$
according to the definition of the total order $<_{\vs}$ on $S$\,.
Hence
$$
(1-t) x_s + t \beta(s) < (1-t) x_{\ts} + t \beta(\ts)\,,
\qquad \forall \quad t \in [0,1]\,.
$$

Condition {\bf C2} is satisfied automatically because
for every configuration in $Y_{|S|}$ we have (\ref{do-k-posle}).

Condition {\bf C3} is also obviously satisfied.

It also follows from the construction that
$$
H(\{(x_s, y_s)\}, t) \in Y_{|S|}
$$
for all $\{(x_s, y_s)\}\in Y_{|S|}$ and
$t \in [0,1]$\,.

Furthermore, it is cleat that $H$ is a deformation
retraction of $Y_{|S|}$ onto the subspace $L$ of
configurations $\{(x_s, y_s)\}\in X_{\vs}$ with
$$
y_{s} = y_1 + s-1\,, \qquad \forall \quad s\in S\,,
$$
$$
x_s=0\,,  \qquad \forall \quad s\in S_{\a}\,,
$$
and
$$
x_s=\beta(s)\,,  \qquad \forall \quad s\in S_{\c}\,.
$$
The subspace $L$ is obviously homeomorphic to the real line $\bbR$\,.

Thus we conclude that $Y_{|S|}$ and hence
$X_{\vs}$ is contractible.

This completes the proof of Proposition \ref{weak-equiv} and
hence the proof of Theorem \ref{SymRbr-braces}.

\begin{Example}
\label{primer}
{\rm
Let us illustrate the proof of contractibility for $X_{\vs}$
with the map
$$
\vs:\{1,2,3,4,5,6\} \to \{\al, \beta, \ga, \de\}
$$
from Example \ref{primer-al-be-ga-de}.
Recall that $c_S=\a$\, $\chi(\al)= \chi(\ga)= \chi(\de) =\c$\,,
and $\chi(\beta)= \a$\,.

The space $X_{\vs}$ consists of configurations from
$\conf(\{\al, \beta, \ga, \de\})$ satisfying the following
conditions:

{\it i}) $x_{\beta}=0 < x_{\ga} < x_{\de}$\,,

{\it ii}) $x_{\al} >0 $\,,

{\it iii}) if $x_\al = x_\ga$ then $y_{\al} < y_{\ga}$\,,

{\it iv}) if $x_\al = x_\de$ then $y_{\al} < y_{\de}$\,.

In the first step of the above proof we retract $X_{\vs}$
onto the subspace $Z_0$ of configurations satisfying the
property
$$
y_{\al} = \min (y_{\al}, y_{\beta}, y_{\ga}, y_{\de})\,.
$$
Second, we retract $Z_0$ to the subspace $Y_1$ of configurations
satisfying in addition the property
$$
\min (y_{\beta}, y_{\ga}, y_{\de}) = y_{\al}+1\,.
$$
Next, we retract $Y_1$ to the subspace $Z_1$ which
consists of configurations $\{(x_s, y_s)\}\in Y_1$
with
$$
y_{\beta} = \min (y_{\beta}, y_{\ga}, y_{\de})\,.
$$

We keep doing so until we get the subspace $Y_4$
of configurations $\{(x_s, y_s)\}\in X_{\vs}$ with
\begin{equation}
\label{zazhal}
y_{\de}=y_{\ga} + 1= y_{\beta} + 2 = y_{\al} + 3\,.
\end{equation}

Then we retract the resulting space $Y_4$
to the subspace $L$ of configurations
$\{(x_s, y_s)\}\in X_{\vs}$ satisfying (\ref{zazhal}) and
$$
x_{\al}=1, \qquad x_{\beta} =0, \qquad x_{\ga}=2, \qquad x_{\de} = 3\,.
$$
Performing the latter retraction
we may need to move horizontally the point labeled
by $\al$ through the vertical lines containing the points labeled by
$\ga$ and $\de$\,. In doing so we will not violate
conditions {\it iii}) and {\it iv})  because the
inequalities $y_{\al}< y_{\ga}$ and $y_{\al}< y_{\de}$
are already achieved at the previous steps.

The subspace $L$ is obviously homeomorphic to the real
line. Thus contractibility of $X_{\vs}$ follows.
}
\end{Example}

\section{Proof of Theorem \ref{th}}\label{proof2}
\label{pfth}
Let us return to the dg (SC) 2-operad $\scR$
introduced in Definition \ref{defi-br} and show that

\begin{Proposition} \label{br-styag}
For every pruned 2-tree (pruned SC 2-tree) $\t$

1) the cochain complex $\scR(\t)$ is contractible;

2) there exist natural identifications
$$
H^0(\scR(\t))=\k
$$
under which all operadic composition
maps of the operad $H^\bullet(\scR)$
evaluated on $1\in \k$ produce $1\in \k$\,.
\end{Proposition}

\begin{proof} Due to Lemma \ref{lmfil1} the inclusion
$$
\scR \hookrightarrow \sc
$$
is a quasi-isomorphism of dg SC 2-operads.
We start with the non-SC case. We have to show that
for every
pruned 2-tree $\t:S\to T$,
 the cochain complex
$$|\seq|(\t)
$$
is contractible. This  was
proved in Proposition 6.4 in \cite{dgcat}.
For the convenience of the reader we briefly recall
the argument.

By definition, $|\seq|(\t)$ is
the realization of the cosimplicial/polysimplicial
set (see Section \ref{seq})
\begin{equation}
\label{c-p-s-set}
\{\{I_s\}_{s\in S } ; J \} \to
\seq(\t)^J_{\{I_s\}_{s\in S }}
\end{equation}
in the category of cochain complexes.

Thus we need to show that realizing (\ref{c-p-s-set})
in the category of topological spaces we get a contractible
space.

For this purpose we fix the ordinal $J$ and
consider the corresponding polysimplicial set
\begin{equation}
\label{p-s-set-J}
\{\{I_s\}_{s\in S }\} \to
\seq(\t)^J_{\{I_s\}_{s\in S }}\,.
\end{equation}
It is shown in \cite{dgcat} that for every (non-empty)
ordinal $J$
\begin{equation}
\label{h-J}
|\seq(\t)^J_{\bullet,\ldots,\bullet}|_{\top}
\cong |\seq(\t)^{[0]}_{\bullet,\ldots,\bullet}|_{\top}
\times \D^J
\end{equation}
and moreover the collection of homeomorphisms (\ref{h-J})
gives an isomorphism of the corresponding cosimplicial
topological spaces. Here $[0]$ is the one element ordinal.

Thus, in order to prove contractibility
of the realization of (\ref{c-p-s-set}) we need
to prove contractibility of the topological space
\begin{equation}
\label{ono}
|\seq(\t)^{[0]}_{\bullet,\ldots,\bullet}|_{\top}\,.
\end{equation}

This space admits the following explicit description.
A point of $|\seq(\t)^{[0]}_{\bullet,\ldots,\bullet}|_{\top}$
is given by an equivalence class of decompositions of the
segment $[0, |S|]$ into a number of subsegments
labeled by elements of $S$. The labeling should satisfy the following
conditions:

$\aleph 1)$ if $s_1, s_2\in S$ and a segment labeled by $s_2$ lies between
segments labeled by $s_1$ then $\t(s_1)>\t(s_2)$ in $T$\,,

$\aleph 2)$ if for $s_1,s_2\in S$ we have
$\t(s_1)= \t(s_2)$ and $s_1 < s_2$ then all segments labeled
by $s_1$ are on the left-hand side of all segments labeled by $s_2$\,,

$\aleph 3)$ for every $s\in S$ the total length of all
segments labeled by $s$ is $1$\,.

Two such decompositions are equivalent if one is obtained from
the other by a number of operations of
the following two types:

a) adding into or deleting from our decomposition a number of labeled segments
of length 0\,,

b) joining two neighboring segments of our decomposition labeled by an
element $s \in S$ into one segment labeled by $s$,
or the inverse operation.

In \cite{dgcat} it was proved, by induction on $|T|$,
that the space (\ref{ono}) is a product of simplices and
hence (\ref{ono}) is contractible. Thus we deduce that so is
the cochain complex $|\seq|(\t)$\,.

We now pass to the  SC-case.
Let  $\t: S\to T$ be a pruned SC 2-tree
with $S= S_{\a}\sqcup S_{\c}$\,, where $S_{\a}$ is
the preimage of the minimal element of $T$ and
$S_{\c} = S\setminus S_{\a}$\,. The subset $S_{\a}$
may, in principle, be empty.

Recall that $\sc(\t)$ is the realization of
the polysimplicial set
\begin{equation}
\label{p-s-set}
\{ \{I_s\}_{s\in S_{\c}}\}~ \to ~
\scseq(\t)_{\{I_s\}_{s\in S_{\c}} }
\end{equation}
in the category of cochain complexes.

Each element $u$ of $\scseq(\t)_{\{I_s\}_{s\in S_{\c} }}$
is a total order $>_u$ on
$$
\I = \bigsqcup\limits_{s\in S_{\c}} I_s
\sqcup S_{\a}
$$
satisfying the following conditions:

--- it agrees with the total order on each $I_{s}$ and
with the total order on $S_{\a}$\,,

--- if $i,k\in I_{s_1}$, $j\in I_{s_2}$,
$s_1\neq s_2$ and $i<_u j<_u k$, then
$\t(s_2)<\t(s_1)$\,,

--- if $s_1,s_2\in S_{\c}$, $s_1<s_2$, and
$\t(s_1)=\t(s_2)$, then all elements of $I_{s_1}$
are strictly smaller than all elements of
$I_{s_2}$.

As well as the space (\ref{ono}) the realization
$|\scseq(\t)|_{\top}$ of (\ref{p-s-set}) has the following
explicit description.
A point of $|\scseq(\t)|_{\top}$
is given by an equivalence class of decompositions of the
segment $[0, |S|]$ into a number of subsegments
labeled by elements of $S$. The labeling should satisfy the following
conditions:

$\aleph 0')$ for each $s\in S_{\a}$ there is exactly one segment labeled by $s$
and its length is $1$; if for $s_1, s_2\in S_{\a}$ we have
$s_1<s_2$ then the segment labeled by $s_1$ is on the left-hand side
of the segment labeled by $s_2$\,.

$\aleph 1')$ if $s_1, s_2\in S_{\c}$ and a segment labeled by $s_2$ lies between
segments labeled by $s_1$ then $\t(s_1)>\t(s_2)$\,,

$\aleph 2')$ if for $s_1,s_2\in S_{\c}$ we have
$\t(s_1)= \t(s_2)$ and $s_1 < s_2$ then all segments labeled
by $s_1$ are on the left-hand side of all segments labeled by $s_2$\,,

$\aleph 3')$ for every $s\in S_{\c}$ the total length of all
segments labeled by $s$ is $1$\,.

Two such decompositions are equivalent if one is obtained from
the other by a number of operations of
the following two types:

a) adding into or deleting from our decomposition a number of labeled segments
of length 0\,,

b) joining two neighboring segments of our decomposition labeled by an
element $s \in S_{\c}$ into one segment labeled by $s$,
or the inverse operation.

If we remove all elements of $S_{\a}$ from $S$ and the minimal
element $t_{min}$ from $T$ then we get a usual pruned (non-SC) 2-tree
\begin{equation}
\label{tree-c}
\wt  = \t \Big|_{S_{\c}} : S_{\c} \to T \setminus \{t_{min}\}\,.
\end{equation}
To this 2-tree we assign the following polysimplicial set
\begin{equation}
\label{p-s-set-0}
\{ \{I_s\}_{s\in S_{\c} } \} \to
\seq(\wt)^{[0]}_{\{I_s\}_{s\in S_{\c} }}
\end{equation}
and the corresponding topological space
\begin{equation}
\label{ono1}
|\seq(\wt)^{[0]}_{\bullet,\ldots,\bullet}|_{\top}
\end{equation}
which was explicitly described above.
(The space (\ref{ono1}) is obtained from the space (\ref{ono})
via replacing $\t$ by $\wt$.)

We have the obvious projection
$$
P: |\scseq(\t)|_{\top} \to
|\seq(\wt)^{[0]}_{\bullet,\ldots,\bullet}|_{\top}
$$
which sends a point of $|\scseq(\t)|_{\top}$ to
a point of $|\seq(\wt)^{[0]}_{\bullet,\ldots,\bullet}|_{\top}$ by
collapsing each segment labeled by an element of $S_{\a}$ to a point.

Conversely, given:

$i)$ a point $x\in
|\seq(\wt)^{[0]}_{\bullet,\ldots,\bullet}|_{\top}$\,, and

$ii)$ a monotonous map $U: S_{\a} \to [0, |S_{\c}|]$\\
one can reconstruct a point in  $|\seq(\t)|_{\top}$ by
inserting unit segments labeled by $s\in S_{\a}$
in the place of the point $U(s)$\,.

Thus we conclude that
$$
|\scseq(\t)|_{\top} \cong
|\seq(\wt)^{[0]}_{\bullet,\ldots,\bullet}|_{\top} \times \D^{|S_{\a}|}\,.
$$
Due to Proposition 6.4 from \cite{dgcat} the first
component $|\seq(\wt)^{[0]}_{\bullet,\ldots,\bullet}|_{\top}$
is contractible. Hence so is $|\scseq(\t)|_{\top}$\,.

Thus we proved that $\sc(\t)$ is contractible for
every pruned SC 2-tree $\t$\,.

The identifications from Part 2) of this proposition
come from the fact that the topological spaces
$$
|\scseq(\t)^{\bul}_{\bullet,\ldots,\bullet}|_{\top}
$$
for pruned 2-trees $\t$ and
$$
|\scseq(\t)_{\bullet,\ldots,\bullet}|_{\top}
$$
for pruned SC 2-trees $\t$ are contractible.
These topological realizations inherit the operadic
compositions, whence Part 2) of this proposition.
\end{proof}

Proposition \ref{br-styag} implies that the cofibrant resolution
$\R \scR$ of $\scR$ is also a cofibrant resolution
of the trivial (SC) 2-operad $\triv$ in the category
of reduced (SC) 2-operads over cochain complexes.

Therefore, due to Batanin's theorem (Theorem \ref{bat3})
the symmetrization $\sym \R \scR$ of $\R \scR$ is quasi-isomorphic
to the singular chain operad of Voronov's
Swiss Cheese operad $\SC_2$\, (in particular, the non-SC part of $\sym \R \scR$ is
quasi-isomorphic to the singular chain operad of the little disc operad
 (Theorem \ref{bat33})).

Due to Theorem \ref{SymRbr-braces} the SC operad
$\sym \R \scR$ is quasi-isomorphic to $\scoR$ which is,
in turn, quasi-isomorphic to the SC operad
$|\se|$ by Lemma \ref{lmfil-se}.

Finally, by construction
the SC operad $|\se|$ is isomorphic to the operad
$|\op|$\,.

Thus we conclude that the two-colored  operad $|\op|$
is  quasi-isomorphic
to the singular chain operad of Voronov's
Swiss Cheese operad $\SC_2$\, (and the non-SC part of $|\op|$ is
quasi-isomoprhic to the singular chain operad of the little disc operad).

It remains to show that the induced action
of  $H_{-\bul}(\SC_2)$ on the pair $(HH^{\bul}(A,A), A)$
coincides with the one given in Proposition \ref{HH-swiss}.
For this purpose we present operations on the pair
\begin{equation}
\label{para}
(\,C^{\bul}(A,A), A\,)
\end{equation}
which come from the action of $|\op|$
and which induce on $(HH^{\bul}(A,A), A)$
the  $H_{-\bul}(\SC_2)$-algebra structure
from Proposition \ref{HH-swiss}\,.

These operations are the cup-product and the Gerstenhaber
bracket \cite{Ger} on $C^{\bul}(A,A)$, the associative product on $A$,
and the following contraction of a cochain $P$ with
elements of the algebra $A$:
\begin{equation}
\label{contraction}
i(P, a) = a\,P(1,1,\dots, 1)~ : ~ C^{\bul}(A,A)\otimes A \to A\,.
\end{equation}
We would like to remark that since $C^{\bul}(A,A)$ is
the normalized Hochschild complex only degree zero cochains
contribute to the contraction.

These operations induce the desired $H_{-\bul}(\SC_2)$-algebra structure
on $(HH^{\bul}(A,A), A)$ and they obviously come from the action of
the SC operad $|\op|$ on the pair (\ref{para}).

Since the cohomology operad $H^{\bul}(|\op|)$ of $|\op|$ is
isomorphic to $H_{-\bul}(\SC_2)$ we conclude that the action of
$|\op|$ on (\ref{para}) induce the desired $H_{-\bul}(\SC_2)$-algebra
structure on $(HH^{\bul}(A,A), A)$\,.

Theorem \ref{th} is proved.  $\Box$

\section*{Appendix}
Let $[n]$ be the standard ordinal $\{0,1,2, \dots, n\}$\,.

Given a collection of $k$ ordinals $[n_1], [n_2], \dots, [n_k]$
we consider the following ordinal
\begin{equation}
\label{cal-I}
\I_{n_1, \dots, n_k} = [n_1]\sqcup [n_2] \sqcup \dots \sqcup [n_k]\,,
\end{equation}
where the order is defined by the following rule:
{\it for $i_1\in [n_{l_1}]$ and $i_2\in [n_{l_2}]$
$i_1 < i_2$ if
\begin{itemize}

\item $l_1 < l_2$ or

\item $l_1 = l_2$ and $i_1 < i_2$ in $[n_{l_1}]$\,.

\end{itemize}
}

Given ordinals $J$, $[n_1], [n_2], \dots, [n_k]$
the collection
\begin{equation}
\label{hom-n-J}
(\Xi_k)^J_{n_1, \dots, n_k} =
\hom_{\Delta}(\I_{n_1, \dots, n_k}, J)
\end{equation}
form a polysimplicial/cosimplicial set.
Indeed $(\Xi_k)^J_{n_1, \dots, n_k}$ is simplicial in
$[n_1], [n_2], \dots, [n_k]$ and cosimplicial
in $J$\,.

In this appendix we show that
\begin{Proposition}
\label{Xi-realize}
The cochain complex $|\Xi_k|$ is concentrated in
nonnegative degrees. Furthermore,
\begin{equation}
\label{Cohomol}
H^{\bul}(|\Xi_k|) =
\begin{cases}
\k \,, ~~ {\rm if} ~ \bul = 0\,, \cr
0\,, ~~ {\rm otherwise}\,.
\end{cases}
\end{equation}
\end{Proposition}

\noindent
{\bf Proof.} The first statement is very easy.
Indeed, an element $v \in \hom_{\Delta}(\I_{n_1, \dots, n_k}, J) $
will not contribute to the realization if it is
degenerate. It is clear that if
$$
|J| <  \sum_{i=1}^k (n_i+1)  - k + 1
$$
then $v$ is degenerate.
Therefore, elements  $v \in \hom_{\Delta}(\I_{n_1, \dots, n_k}, J) $
with
$$
|J|-1 - \sum_{i=1}^k n_i < 0
$$
will not contribute to the realization. Hence the
cochain complex $|\Xi_k|$ is indeed concentrated in
nonnegative degrees.

The cochain complex $| \Xi_k |$ can be considered as
bicomplex
\begin{equation}
\label{bicomplex}
|\Xi_k| = |\Xi_k|^{\bul, \bul}\,.
\end{equation}
The first degree is the total degree of the
simplicial indices. According to our conventions this degree
is nonpositive.  The second degree is the degree
in the cosimplicial index and this degree is nonnegative.
Let us denote by $\pa^s$ the part of the differential in $|\Xi_k|$
which comes from the simplicial indices and by
$\pa^c$ the part of the differential in $|\Xi_k|$ coming from the
cosimplicial structure.

Fixing the second degree we get the cochain
complex
\begin{equation}
\label{Xi-k-m}
|\Xi_k|^{\bul, m}
\end{equation}
which is the realization of the polysimplicial
set
\begin{equation}
\label{hom-m}
([n_1], [n_2], \dots, [n_k])~~ \to ~~
\hom_{\Delta}(\I_{n_1, n_2,  \dots, n_k}, [m])\,.
\end{equation}

It is not hard to see that the realization of
(\ref{hom-m}) in the category of topological spaces
is the following stretched $m$-simplex:
$$
\{(x_0, x_1, \dots, x_m)
\quad | \quad x_i\ge 0\,, \quad x_0 + x_1+ x_2 + \dots + x_m = k\}\,.
$$
Therefore for each $m$ the complex $|\Xi_k|^{\bul, m}$ has non-trivial cohomology
only in degree $0$ and
\begin{equation}
\label{H-0-Xi-m}
H^0(|\Xi_k|^{\bul, m}) = \k\,.
\end{equation}
The class which generates $H^0(|\Xi_k|^{\bul, m})$ is represented by
the map
\begin{equation}
\label{cocycle}
c\in \hom_{\Delta}(\I_{0, \dots, 0}, [m])\,,
\end{equation}
which sends all elements of $\I_{0, \dots, 0}$
to the same element $0\in [m]$\,.
All other maps in $\hom_{\Delta}(\I_{0, \dots, 0}, [m])$
are cohomologous to the cocycle (\ref{cocycle}).

It is not hard to see that
\begin{equation}
\label{Te}
\Te = \bigoplus_{q < 0} |\Xi_k|^{q,\bul} ~ \oplus ~
\pa^{s} (\, |\Xi_k|^{-1,\bul} \,)
\end{equation}
is a subcomplex of the bicomplex $|\Xi_k|$\,.

Equation (\ref{H-0-Xi-m}) implies that each term of
the quotient complex $|\Xi_k|/\Te$ is $\k$\,. Using the
explicit cocycle (\ref{cocycle}) it is not hard to see
that the quotient complex $|\Xi_k|/\Te$ is
$$
\k \stackrel{0}{\to} \k  \stackrel{{\rm id}}{\to} \k
 \stackrel{0}{\to} \k  \stackrel{{\rm id}}{\to} \k
 \stackrel{0}{\to} \dots
$$
and hence
\begin{equation}
\label{Xi-Te}
H^{\bul}(|\Xi_k|\, / \, \Te) =
\begin{cases}
\k \,, ~~ {\rm if} ~ \bul = 0\,, \cr
0\,, ~~ {\rm otherwise}\,.
\end{cases}
\end{equation}

We see from the construction that the bicomplex
$\Te$ (\ref{Te}) is acyclic in the first degree.
Therefore $\Te$ is acyclic as the total complex.

Thus $H^{\bul}(|\Xi_k|)= H^{\bul}(|\Xi_k|\, / \, \Te)$
and the proposition follows.  $\Box$

\vspace{0.5cm}

\noindent\textsc{Department of Mathematics,
University of California at Riverside, \\
900 Big Springs Drive,\\
Riverside, CA 92521, USA \\
\emph{E-mail address:} {\bf vald@math.ucr.edu}}

\vspace{0.5cm}

\noindent\textsc{Mathematics Department,
Northwestern University, \\
2033 Sheridan Rd.,\\
Evanston, IL 60208, USA \\
\emph{E-mail addresses:} {\bf tamarkin@math.northwestern.edu},
{\bf tsygan@math.northwestern.edu}}

\end{document}